\documentclass[oneside,english]{article}
\usepackage[T1]{fontenc}
\usepackage[latin9]{inputenc}
\usepackage{float}
\usepackage{epic,eepic}
\usepackage{comment}
\usepackage{amstext,amsthm,amssymb,amsmath}
\usepackage{epsfig,subfigure,epstopdf}
\usepackage{graphicx}
\usepackage{subfigure}
\usepackage{wrapfig}
\usepackage{fullpage}
\usepackage{color}
\usepackage{multirow}
\usepackage{tabulary}
\usepackage{booktabs}
\usepackage{enumerate}
\usepackage{color}

\theoremstyle{definition}
\newtheorem{theorem}{Theorem}[section]
\newtheorem{lemma}{Lemma}[section]

\makeatletter
\numberwithin{equation}{section}
\numberwithin{figure}{section}

\makeatother

\title{Contrast-independent partially explicit time discretizations for multiscale flow problems. }
\author{Eric T. Chung \footnote{Department of Mathematics, The Chinese University of Hong Kong (CUHK), Hong Kong SAR}, ~Yalchin Efendiev\footnote{Department of Mathematics, Texas A\&M University, College Station, TX 77843, USA \& North-Eastern Federal University, Yakutsk, Russia}, ~ Wing Tat Leung\footnote{Department of Mathematics, University of California, Irvine, USA}, ~ Petr N. Vabishchevich\footnote{Nuclear Safety Institute, Russian Academy of Sciences, Moscow, Russia \& North-Eastern Federal University, Yakutsk, Russia}}

\usepackage{babel}
\begin{document}
\maketitle

\section*{Abstract}
Many multiscale problems have a high contrast, which is expressed as a
very large ratio between the media properties. The contrast is known
to introduce many challenges in the design of multiscale methods and 
domain decomposition approaches. These issues to some extend are analyzed 
in the design of spatial multiscale and domain decomposition approaches.
However, some of 
these issues remain open for time dependent problems as the contrast
affects the time scales, particularly, for explicit methods. 
For example, in parabolic equations, the time step is 
$dt=H^2/\kappa_{max}$, where $\kappa_{max}$ is the largest 
diffusivity.
In this
paper, we address this issue in the context of parabolic equation
by designing a splitting algorithm. The proposed splitting algorithm
treats dominant multiscale modes in the implicit fashion, while 
the rest in the explicit fashion. The unconditional stability of
these algorithms require a special multiscale space design, which is
the main purpose of the paper. We show that with an appropriate choice
of multiscale spaces we can achieve an unconditional stability with
respect to the contrast. This could provide 
computational savings as the time step in explicit methods is
adversely affected by the contrast. 
We discuss some theoretical aspects of the proposed algorithms.
Numerical results are presented.

\section{Introduction}

Many problems have multiple scales and high contrast. Examples include
flows in porous media, composite materials, and so on. In these problems,
one typically observes a large jump in media properties, which is usually
referred as a high contrast, where the contrast is the ratio between
largest and smallest media property values, e.g., diffusivity in
the case of isotropic diffusion in the media. These problems pose
challenges in numerical simulations. Some of these challenges 
in the context of spatial treatments have been addressed
(e.g., \cite{chung2018constraint, NLMC}).

 It is known that the high 
contrast requires special treatment in multiscale methods by
introducing additional multiscale basis functions 
\cite{chung2018constraint, NLMC}. 
The high contrast introduces challenges in temporal discretization,
in particular, for explicit methods. The high contrast in 
media properties introduces stiffness for the systems and requires
small time stepping, particularly, for explicit methods. 
For example, for the parabolic equation with the diffusion 
coefficients $\kappa(x)$, the time stepping for explicit methods
needs to be $H^2/\kappa_{max}$, where $\kappa_{max}$ is the largest
diffusion coefficient. To overcome this difficulty, we introduce
a splitting method, which splits the space and time in an appropriate
way. The resulting discretization's stability is independent of contrast.
This provides a computational savings since the contrast can be very large.

Next, we give some overview of multiscale methods, in particular,
their treatment of the contrast in the context of steady state
problems. 
Multiscale spatial
algorithms have been studied in the literature.
In previous findings, the algorithms, such as
homogenization-based approaches \cite{eh09,le2014msfem}, 
multiscale
finite element methods \cite{eh09,hw97,jennylt03}, 
generalized multiscale finite element methods (GMsFEM) \cite{chung2016adaptiveJCP,MixedGMsFEM,WaveGMsFEM,chung2018fast,GMsFEM13}, 
constraint energy minimizing GMsFEM (CEM-GMsFEM) 
\cite{chung2018constraint, chung2018constraintmixed}, nonlocal
multi-continua (NLMC) approaches \cite{NLMC},
metric-based upscaling \cite{oz06_1}, heterogeneous multiscale method 
\cite{ee03}, localized orthogonal decomposition (LOD) 
\cite{henning2012localized}, equation-free approaches \cite{rk07,skr06}, 
multiscale stochastic approaches \cite{hou2017exploring, hou2019model, hou2018adaptive},
and hierarchical multiscale method \cite{brown2013efficient},
are developed to address spatial heterogeneities.
For high-contrast problems, approaches such as GMsFEM and NLMC, are proposed.
As we mentioned earlier, in porous media applications, the spatial 
heterogeneities
are too complex and have high contrast. 
For this reason, for GMsFEM and related approaches
\cite{chung2018constraint}, multiple basis
functions or continua are designed to capture the multiscale features
due to high contrast
\cite{chung2018constraintmixed, NLMC}. 
These approaches require a careful design of multiscale
dominant modes. The contrast, as it is known, introduces a stiffness
in the dynamical systems. When treating explicitly, one needs to take
very small time steps. In this paper, we will propose an approach
that allows taking the time step to be independent of the contrast.


Our approaches take their origin in splitting algorithms
\cite{marchuk1990splitting,VabishchevichAdditive},
which are initially designed to split various physics.
For example, for 
convection-diffusion equations, these approaches are often used to 
split convection and diffusion.
In these cases, the operator is decomposed
based on physical processes.
In our recent works, we have proposed several approaches for temporal splitting
that uses multiscale spaces  \cite{efendiev2020splitting, efendiev_split2020}.
In  \cite{efendiev2020splitting}, a general framework is proposed where
the transition to simpler problems is carried out 
based on spatial decomposition of the solution.
In \cite{efendiev_split2020}, 
we combine the temporal splitting algorithms and spatial
multiscale methods. We divide the spatial space into various components
and use these subspaces in the temporal splitting. As a result, smaller
systems are inverted in each time step, which reduces the computational
cost. 
These algorithms are implicit, and we prove that they are unconditionally
stable. 
These approaches share some common concepts with IMEX methods
(e.g., \cite{ascher1997implicit}).
There are many approaches for treating multiscale 
stiff systems 
(e.g., \cite{li2008effectiveness,abdulle2012explicit,engquist2005heterogeneous,ariel2009multiscale}).
Our proposed approaches differ from these approaches Our goal is to use
splitting concepts and treat implicitly and explicitly some
parts of the solution in order to make the time step contrast 
independent.

In the paper, we introduce a special multiscale decomposition
and a temporal splitting, which provides a contrast-independent
time discretization for multiscale flow problems.
We consider a parabolic equation with multiscale and high contrast
coefficients. As in our previous CEM-GMsFEM approaches, we select 
dominant basis functions which capture important degrees of freedom
and it is known to give contrast-independent convergence that scales with
the mesh size. We design and
introduce an additional space in the complement space
and these degrees of freedom are treated explicitly.
In typical situations, one has very few degrees of freedom
in dominant basis functions that are treated implicitly.
Thus, the resulting implicit-explicit schemes has very small
implicit part. We show that with our specially designed spaces
the proposed temporal discretization is stable with the time stepping
that is independent of the contrast. We propose several choices for 
multiscale space decomposition. We note that a special decomposition
is needed to remove the contrast in the time stepping, which is
shown in this paper.

We remark several important observations. First, we note that the use
of additional degrees of freedom (basis functions
beyond CEM-GMsFEM basis functions) is needed for dynamic problems,
in general, to handle
missing information. This is even though CEM-GMsFEM can provide accurate
solution for some parabolic equations, the basis functions are computed
based on steady-state information and additional degrees of freedom
are needed to improve solution adaptively. Secondly, our approaches
share some similarities with online methods (e.g., \cite{chung2015residual}), 
where additional
basis functions are added and iterations are performed. The main
difference is that our focus is to find spaces that can provide the
time step to be independent of the contrast for explicit methods. Finally,
we note that restrictive time step (e.g., $dt=H^2$) scales as the coarse mesh
size and thus much coarser.

We present several representative numerical results. We compare
various methods and show that the proposed methods provide a good
approximation with the time step that is independent of the contrast.
We select examples where additional basis functions provide an improvement
by choosing ``singular'' source terms.

The paper is organized as follows. In the next section, we present
Preliminaries. In section 3, we present a general construction
of partially explicit methods. Section 4 is devoted to the construction
of multiscale spaces. We present numerical results in Section 5.
The conclusions are presented in Section 6.

\section{Preliminaries}

We consider the following problem. Find $u$ such that 
\begin{equation}
\label{eq:main}
u_{t}=\nabla\cdot(\kappa\nabla u)\;\text{in }\Omega,
\end{equation}
where $\kappa\in L^{\infty}(\Omega)$ is a high contrast parameter.
We can write the problem in the weak formulation: find $u(t,\cdot)\in V$
such that 
\begin{equation}
(u_{t},v)+a(u,v)=0\;\forall v\in V.
\label{eq:problem_weak}
\end{equation}
In our case $V = H^1_0(\Omega)$,
\[
 (u,v) = \int_{\Omega} u v ,
 \quad \|u\| = (u, u)^{1/2} . 
\]
The bilinear form $a(\cdot, \cdot)$ is given
\[
 a(u,v) = \int_{\Omega} \kappa \nabla u \cdot \nabla v ,
\] 
where $0 < \kappa_{min} \leq \kappa(x) \leq \kappa_{max}, \ x \in \Omega$. 
For energy norm, we have $\|u\|_a = a(u,u)^{1/2}$.

Cauchy problem consists of finding
 $w(t)$ in  $\mathcal{V}$ and $0 < t \leq T$, such that
\begin{equation}\label{2.1}
 \frac{d}{d t} (w(t), v) + a(w,v) = 0 \quad \forall v \in \mathcal{V} , 
 \quad 0 < t \leq T , 
\end{equation} 
and initial condition
\begin{equation}\label{2.2}
 w(0) = w^0 .
\end{equation}

Semi-discretization in space of 
$u(t) \in V$, where $V$ is a finite dimensional subspace of
 $\mathcal{V}$ ($V \subset \mathcal{V}$),
such that
\begin{equation}\label{2.3}
 \frac{d}{d t} (u(t), v) + a(u,v) = 0 \quad \forall v \in V, 
 \quad 0 < t \leq T , 
\end{equation} 
\begin{equation}\label{2.4}
 u(0) = u^0 .
\end{equation}
Taking
$v = d u /dt$ in (\ref{2.3}) we get
\begin{equation}\label{2.5}
 \|u(t)\|_a \leq \|u^0\|_a ,
 \quad 0 < t \leq T . 
\end{equation} 

For simplicity, we consider a fixed time step,  $\tau$ and
$t^n = n \tau, \ n =  0, \ldots, N, \ N \tau = T$, $u^n = u(t^n)$.
It is known that 
(e.g., \cite{SamarskiiTheory,SamarskiiMatusVabischevich2002}), 
in the class of two-level schemes, implicit methods (backward Euler)
is unconditionally stable, and a forward method (forward Euler) conditionally 
stable. 

When using implicit method, 
 $u_H^{n+1} \approx u(t^n)$ ($H$ is the coarse mesh size)
\begin{equation}\label{2.6}
 \left ( \frac{u_H^{n+1}  - u_H^{n}}{\tau }, v \right ) + a(u_H^{n+1} ,v) = 0 
\quad \forall v \in V, 
 \quad n = 0, \ldots, N-1.
\end{equation} 
We take in (\ref{2.6}), $v = 2(u_H^{n+1} -u_H^{n})$ 
and into account the symmetry in the bilinear form
$a(\cdot, \cdot)$, we have
\[
 \frac{2}{\tau } \|u_H^{n+1} - u_H^{n+1}\|^2 + \|u_H^{n+1} - u_H^{n+1}\|_a^2 + \|u_H^{n+1}\|_a^2  - \|u_H^{n}\|_a^2 = 0 .
\] 
Thus,
\begin{equation}\label{2.8}
 \|u_H^{n}\|_a \leq \|u_H^0\|_a , 
 \quad n = 1, \ldots, N,  
\end{equation}
which is a discrete version of (\ref{2.5}).
The estimate
(\ref{2.8}) guarantees the unconditional stability.

The stability condition for the explicit scheme
\begin{equation}\label{2.9}
 \left ( \frac{u_H^{n+1}  - u_H^{n}}{\tau }, v \right ) + a(u_H^{n} ,v) = 0 \quad \forall v \in V, 
 \quad n = 0, \ldots, N-1
\end{equation} 
is carried out by 
taking into account
$v = 2(u_H^{n+1} -u_H^{n})$  in (\ref{2.9})
and after some manipulations, we have
\[
 \frac{2}{\tau } \|u_H^{n+1} - u_H^{n}\|^2 - \|u_H^{n+1} - u_H^{n}\|_a^2 + \|u_H^{n+1}\|_a^2  - \|u_H^{n}\|_a^2 = 0.
\]    
Thus, the  stability
(the estimate (\ref{2.8})) will take the place if 
\begin{equation}\label{2.10}
 \|v\|^2 \geq \frac{\tau}{2} \|v\|_a^2 \quad \forall v \in V .
\end{equation} 

\section{Partially Explicit Temporal Splitting Scheme}

In this section, we first introduce a general partial splitting
algorithm
for  $u_{H}$ of problem (\ref{eq:main})
defined as 
\begin{align*}
(u_{H,t},v) & +a(u_{H},v)=0\;\forall v\in V_{H},
\end{align*}
where $V_{H}$ is a coarse grid finite element space. We consider
$V_{H}$ can be decomposed into two subspaces $V_{H,1}$ and $V_{H,2}$,
namely, 
\[
V_{H}=V_{H,1}+ V_{H,2}.
\]

We will use a time discretization scheme: 
finding $\{u_{H,1}^{n}\}_{n=1}^{N}\in V_{1,H},\;\{u_{H,2}^{n}\}_{n=1}^{N}\in V_{H,2}$
\begin{equation}
\begin{split}
(u_{H,1}^{n+1}-u_{H,1}^{n},v)+\mu(u_{H,2}^{n+1}-u_{H,2}^{n},v)+
(1-\mu)(u_{H,1}^{n}-u_{H,1}^{n-1},v)=\\
-\tau a(u_{H,1}^{n+1}+u_{H,2}^{n},v),\;v\in V_{1,H}\\
(u_{H,2}^{n+1}-u_{H,2}^{n},v) +\mu(u_{H,1}^{n+1}-u_{H,1}^{n},v)+
(1-\mu)(u_{H,2}^{n}-u_{H,2}^{n-1},v)=\\
-\tau a((1-\omega)u_{H,1}^{n}+\omega u_{H,1}^{n+1}+u_{H,2}^{n},v),\;v\in V_{2,H}.
\end{split}
\end{equation}
Here, we will consider some options for $\mu$ and $\omega$ in $[0,1]$ and
spaces $V_{H,1}$ and $V_{H,2}$.
We note that if $\mu=0$ and $\omega=0$, 
the second equation does not require $u_{H,1}^{n+1}$ and is totally decoupled.
When  $\mu=0$ and $\omega=1$,  the equations can be solved sequentially 
(the second equation
is solved after solving the first equation).
When  $\mu=1$,  equations need to be solved together at each new time step.

As a first case, we briefly consider a case $\mu=1$ and $\omega=1$
and 
$V_{H,1}$ and $V_{H,2}$ as two orthogonal spaces such that
\[
 (v_1, v_2) = 0 \quad \forall v_1 \in V_1, \ \forall v_2 \in V_2 .
\] 
The scheme  can be simplified as
\begin{equation}\label{3.2}
\begin{split}
 \left ( \frac{u_{H,1}^{n+1}  - u_{H,1}^{n}}{\tau }, v_1 \right ) & + a(u_{H,1}^{n+1}+u_{H,2}^{n} ,v_1) = 0 \quad \forall v_1 \in V_1, \\
 \left ( \frac{u_{H,2}^{n+1}  - u_{H,2}^{n}}{\tau }, v_2 \right ) & + a(u_{H,1}^{n+1}+u_{H,2}^{n} ,v_2) = 0 \quad \forall v_2 \in V_2, \\
 & \quad n = 0, \ldots, N-1 . 
\end{split} 
\end{equation} 
Initial conditions are mapped in corresponding spaces accordingly.

\begin{theorem}
The partial explicit scheme
(\ref{3.2}) is stable if 
\begin{equation}\label{3.4}
 \|v_2\|^2 \geq \frac{\tau}{2} \|v_2\|_a^2  \quad \forall v_2 \in V_2 .
\end{equation}
Under these conditions, we have
\begin{equation}\label{3.5}
 \|u_{H}^n\|_a \leq \|u_H^0\|_a , 
 \quad u_{H}^n = u_{H,1}^n+u_{H,2}^n ,
 \quad n = 1, \ldots, N . 
\end{equation} 
\end{theorem}

\begin{proof}
We have the following identity
\[
 u_{H,1}^{n+1}+u_{H,1}^{n} = \frac{1}{2} (u_{H}^{n+1}+u_{H}^{n}) + 
\frac{1}{2} (u_{H,1}^{n+1}-u_{H,1}^{n}) 
 - \frac{1}{2} (u_{H,2}^{n+1}-u_{H,2}^{n}) .
\]    
We take in (\ref{3.2}) 
\[
 v_1 = 2 (u_{H,1}^{n+1}-u_{H,1}^{n}),
 \quad v_2 = 2 (u_{H,2}^{n+1}-u_{H,2}^{n}) .
\] 
Summing two equations, we have
\begin{equation}
\begin{split}
 \frac{2}{\tau} \|u_{H,1}^{n+1}-u_{H,1}^{n}\|^2 & +  \frac{2}{\tau} \|u_{H,2}^{n+1}-u_{H,2}^{n}\|^2 + a(u_H^{n+1}+u_H^{n},u_H^{n+1}-u_H^{n}) \\
 & + \|u_{H,1}^{n+1}-u_{H,1}^{n}\|_a^2 - \|u_{H,2}^{n+1}-u_{H,2}^{n}\|_a^2 = 0.
\end{split}
\end{equation}
If
(\ref{3.4})
holds,
we get the estimate
(\ref{3.5}).
Note that the main finding consists of the constraint on the time step that
is due to the explicit part of the scheme.
\end{proof}

Next, we assume that the
spaces $V_{H,1}$ and $V_{H,2}$ are not necessarily orthogonal
and take $\mu=0$.
We thus obtain the following time discretization scheme: finding $\{u_{H,1}^{n}\}_{n=1}^{N}\in V_{1,H},\;\{u_{H,2}^{n}\}_{n=1}^{N}\in V_{H,2}$
\begin{align}
(u_{H,1}^{n+1},v) & =(u_{H,1}^{n},v)-(u_{H,2}^{n}-u_{H,2}^{n-1},v)-\tau a(u_{H,1}^{n+1}+u_{H,2}^{n},v),\;v\in V_{1,H} \label{eq:sch1} \\
(u_{H,2}^{n+1},v) & =(u_{H,2}^{n},v)-(u_{H,1}^{n}-u_{H,1}^{n-1},v)-\tau a((1-\omega)u_{H,1}^{n}+\omega u_{H,1}^{n+1}+u_{H,2}^{n},v),\;v\in V_{2,H}. \label{eq:sch2}
\end{align}
The numerical solution $u_{H}\in V_{H}$ is the sum of $u_{H,1}$
and $u_{H,2}$, $u_{H}=u_{H,1}+u_{H,2}$, as before.

Now, we will prove stability of the scheme (\ref{eq:sch1})-(\ref{eq:sch2}). 
To do so, we recall the strengthened Cauchy Schwarz inequality \cite{aldaz2013strengthened}. Let $S_1$ and $S_2$ be finite dimensional spaces with $S_1\cap S_2 = \{0\}$. Then
there is a constant $0< \beta_0 < 1$ such that
\begin{equation*}
(s_1,s_2) \leq \beta_0 \|s_1\| \|s_2\|
\end{equation*} 
where $\beta_0$ depends on $S_1$ and $S_2$. 
So, there is a constant $\gamma$, depending on $V_{H,1}$ and $V_{H,2}$, such that
\begin{equation}
\label{eq:gamma}
\gamma:=\sup_{v_{1}\in V_{H,1},v_{2}\in V_{H,2}}\cfrac{(v_{1},v_{2})}{\|v_{1}\| \|v_{2}\|}<1.
\end{equation}


\begin{theorem}
\label{thm32}
The partially explicit scheme (\ref{eq:sch1})-(\ref{eq:sch2}) is stable if 
\begin{equation}
\tau\sup_{v\in V_{H,2}}\cfrac{\|v\|_{a}^{2}}{\|v\|^{2}}\leq\cfrac{1-\gamma^{2}}{(2-\omega)}.\label{eq:stab_cond}
\end{equation}
where $\gamma$ is defined in (\ref{eq:gamma}). 
Moreover, we have the following stability estimate
\[
\cfrac{\gamma^{2}}{2}\sum_{i=1,2}\|u_{H,i}^{n+1}-u_{H,i}^{n}\|^{2}+\cfrac{\tau}{2}\|u_{H}^{n+1}\|_{a}^{2}\leq\cfrac{\gamma^{2}}{2}\sum_{i=1,2}\|u_{H,i}^{n}-u_{H,i}^{n-1}\|^{2}+\cfrac{\tau}{2}\|u_{H}^{n}\|_{a}^{2},\;\text{for }n\geq1.
\]
\end{theorem}

\begin{proof}
By (\ref{eq:sch1}) and (\ref{eq:sch2}), we have 
\begin{align}
(u_{H,1}^{n+1}-u_{H,1}^{n}+u_{H,2}^{n}-u_{H,2}^{n-1},v) & =-\tau a(u_{H,1}^{n+1}+u_{H,2}^{n},v),\;v\in V_{1,H}\label{eq:sch3}\\
(u_{H,2}^{n+1}-u_{H,2}^{n}+u_{H,1}^{n}-u_{H,1}^{n-1},v) & =-\tau a((1-\omega)u_{H,1}^{n}+\omega u_{H,1}^{n+1}+u_{H,2}^{n},v),\;v\in V_{2,H}.\label{eq:sch4}
\end{align}
Taking $v=u_{H,1}^{n+1}-u_{H,1}^{n}$ in (\ref{eq:sch3}) and $v=u_{H,2}^{n+1}-u_{H,2}^{n}$
in (\ref{eq:sch4}), we obtain 
\begin{align}
(u_{H,1}^{n+1}-u_{H,1}^{n}+u_{H,2}^{n}-u_{H,2}^{n-1},u_{H,1}^{n+1}-u_{H,1}^{n}) & =-\tau a(u_{H,1}^{n+1}+u_{H,2}^{n},u_{H,1}^{n+1}-u_{H,1}^{n}),\label{eq:sch5}\\
(u_{H,2}^{n+1}-u_{H,2}^{n}+u_{H,1}^{n}-u_{H,1}^{n-1},u_{H,2}^{n+1}-u_{H,2}^{n}) & =-\tau a((1-\omega)u_{H,1}^{n}+\omega u_{H,1}^{n+1}+u_{H,2}^{n},u_{H,2}^{n+1}-u_{H,2}^{n}).\label{eq:sch6}
\end{align}
The left hand side of the equation (\ref{eq:sch5}) can be estimated
in the following way 
\begin{align*}
(u_{H,1}^{n+1}-u_{H,1}^{n}+u_{H,2}^{n}-u_{H,2}^{n-1},u_{H,1}^{n+1}-u_{H,1}^{n}) & =\|u_{H,1}^{n+1}-u_{H,1}^{n}\|^{2}+(u_{H,2}^{n}-u_{H,2}^{n-1},u_{H,1}^{n+1}-u_{H,1}^{n})\\
 & \geq\|u_{H,1}^{n+1}-u_{H,1}^{n}\|^{2}-\gamma\|u_{H,2}^{n}-u_{H,2}^{n-1}\|\|u_{H,1}^{n+1}-u_{H,1}^{n}\|\\
 & \geq\cfrac{1}{2}\|u_{H,1}^{n+1}-u_{H,1}^{n}\|^{2}-\cfrac{\gamma^{2}}{2}\|u_{H,2}^{n}-u_{H,2}^{n-1}\|^{2}.
\end{align*}
Similarly, the left hand side of the equation (\ref{eq:sch6}) can be estimated as follows
\begin{align*}
(u_{H,2}^{n+1}-u_{H,2}^{n}+u_{H,1}^{n}-u_{H,1}^{n-1},u_{H,2}^{n+1}-u_{H,2}^{n}) & \geq\cfrac{1}{2}\|u_{H,2}^{n+1}-u_{H,2}^{n}\|^{2}-\cfrac{\gamma^{2}}{2}\|u_{H,1}^{n}-u_{H,1}^{n-1}\|^{2}.
\end{align*}
To compute the sum of the right hand sides of (\ref{eq:sch5}) and
(\ref{eq:sch6}), we notice that 
\begin{align*}
 & -a(u_{H,1}^{n+1}+u_{H,2}^{n},u_{H,1}^{n+1}-u_{H,1}^{n})-a((1-\omega)u_{H,1}^{n}+\omega u_{H,1}^{n+1}+u_{H,2}^{n},u_{H,2}^{n+1}-u_{H,2}^{n})\\
= & -\omega a(u_{H,1}^{n+1}+u_{H,2}^{n},u_{H}^{n+1}-u_{H}^{n})-(1-\omega)\Big(a(u_{H,1}^{n+1},u_{H,1}^{n+1}-u_{H,1}^{n})+a(u_{H,2}^{n},u_{H,2}^{n+1}-u_{H,2}^{n})\Big)\\
 & -(1-\omega)\Big(a(u_{H,1}^{n},u_{H,2}^{n+1}-u_{H,2}^{n})+a(u_{H,2}^{n},u_{H,1}^{n+1}-u_{H,1}^{n})\Big).
\end{align*}
We will first estimate the term $-\omega a(u_{H,1}^{n+1}+u_{H,2}^{n},u_{H}^{n+1}-u_{H}^{n})$ as follows:
\begin{equation}
\label{eq:sch7}
-a(u_{H,1}^{n+1}+u_{H,2}^{n},u_{H}^{n+1}-u_{H}^{n})=-a(u_{H}^{n+1},u_{H}^{n+1}-u_{H}^{n})+a(u_{H,2}^{n+1}-u_{H,2}^{n},u_{H}^{n+1}-u_{H}^{n}).
\end{equation}
We then have 
\[
-a(u_{H}^{n+1},u_{H}^{n+1}-u_{H}^{n})=\cfrac{1}{2}\Big(\|u_{H}^{n}\|_{a}^{2}-\|u_{H}^{n+1}-u_{H}^{n}\|_{a}^{2}-\|u_{H}^{n+1}\|_{a}^{2}\Big)
\]
and
\[
a(u_{H,2}^{n+1}-u_{H,2}^{n},u_{H}^{n+1}-u_{H}^{n})\leq\cfrac{1}{2}\|u_{H}^{n+1}-u_{H}^{n}\|_{a}^{2}+\cfrac{1}{2}\|u_{H,2}^{n+1}-u_{H,2}^{n}\|_{a}^{2}
\]

\[
-\tau\omega a(u_{H,1}^{n+1}+u_{H,2}^{n},u_{H}^{n+1}-u_{H}^{n})\leq\cfrac{\tau\omega}{2}\Big(\|u_{H}^{n}\|_{a}^{2}-\|u_{H}^{n+1}\|_{a}^{2}+\|u_{H,2}^{n+1}-u_{H,2}^{n}\|_{a}^{2}\Big).
\]
Next, we have 
\[
-\tau(1-\omega)a(u_{H,1}^{n+1},u_{H,1}^{n+1}-u_{H,1}^{n})=\cfrac{\tau(1-\omega)}{2}\Big(\|u_{H,1}^{n}\|_{a}^{2}-\|u_{H,1}^{n+1}-u_{H,1}^{n}\|_{a}^{2}-\|u_{H,1}^{n+1}\|_{a}^{2}\Big)
\]
and

\[
-\tau(1-\omega)a(u_{H,2}^{n},u_{H,2}^{n+1}-u_{H,2}^{n})=\cfrac{\tau(1-\omega)}{2}\Big(\|u_{H,2}^{n}\|_{a}^{2}+\|u_{H,2}^{n+1}-u_{H,2}^{n}\|_{a}^{2}-\|u_{H,2}^{n+1}\|_{a}^{2}\Big).
\]
So, the right hand side of (\ref{eq:sch7}) becomes
\begin{align*}
 & -a(u_{H,1}^{n},u_{H,2}^{n+1}-u_{H,2}^{n})-a(u_{H,2}^{n},u_{H,1}^{n+1}-u_{H,1}^{n})\\
= & -a(u_{H,1}^{n},u_{H,2}^{n+1})-a(u_{H,2}^{n},u_{H,1}^{n+1})+2a(u_{H,2}^{n},u_{H,1}^{n})\\
= & a(u_{H,1}^{n+1}-u_{H,1}^{n},u_{H,2}^{n+1}-u_{H,2}^{n})-a(u_{H,1}^{n+1},u_{H,2}^{n+1})+a(u_{H,2}^{n},u_{H,1}^{n}).
\end{align*}
Note that,
\begin{align*}
a(u_{H,1}^{n+1}-u_{H,1}^{n},u_{H,2}^{n+1}-u_{H,2}^{n}) & \leq\|u_{H,1}^{n+1}-u_{H,1}^{n}\|_{a}\|u_{H,2}^{n+1}-u_{H,2}^{n}\|_{a}\\
 & \leq\cfrac{1}{2}\|u_{H,1}^{n+1}-u_{H,1}^{n}\|_{a}^{2}+\cfrac{1}{2}\|u_{H,2}^{n+1}-u_{H,2}^{n}\|_{a}^{2}
\end{align*}
Hence, we have 
\begin{align*}
 & -\tau(1-\omega)\Big(a(u_{H,1}^{n+1},u_{H,1}^{n+1}-u_{H,1}^{n})+a(u_{H,2}^{n},u_{H,2}^{n+1}-u_{H,2}^{n})+a(u_{H,1}^{n},u_{H,2}^{n+1}-u_{H,2}^{n})+a(u_{H,2}^{n},u_{H,1}^{n+1}-u_{H,1}^{n})\Big)\\
\leq & \cfrac{\tau(1-\omega)}{2}\Big(\|u_{H,1}^{n}\|_{a}^{2}+\|u_{H,2}^{n}\|_{a}^{2}-\|u_{H,1}^{n+1}\|_{a}^{2}-\|u_{H,2}^{n+1}\|_{a}^{2}+2\|u_{H,2}^{n+1}-u_{H,2}^{n}\|_{a}^{2}-2a(u_{H,1}^{n+1},u_{H,2}^{n+1})+2a(u_{H,2}^{n},u_{H,1}^{n})\Big)\\
= & \cfrac{\tau(1-\omega)}{2}\Big(\|u_{H}^{n}\|_{a}^{2}-\|u_{H}^{n+1}\|_{a}^{2}+2\|u_{H,2}^{n+1}-u_{H,2}^{n}\|_{a}^{2}\Big).
\end{align*}
Combining the above results, 
\begin{align*}
 & \cfrac{\gamma^{2}}{2}\sum_{i=1,2}\|u_{H,i}^{n+1}-u_{H,i}^{n}\|_{L^{2}}^{2}+\cfrac{1-\gamma^{2}}{2}\sum_{i=1,2}\|u_{H,i}^{n+1}-u_{H,i}^{n}\|_{L^{2}}^{2}+\cfrac{\tau}{2}\|u_{H}^{n+1}\|_{a}^{2}\\
\leq & \cfrac{\gamma^{2}}{2}\sum_{i=1,2}\|u_{H,i}^{n}-u_{H,i}^{n-1}\|_{L^{2}}^{2}+\cfrac{\tau}{2}(2-\omega)\|u_{H,2}^{n+1}-u_{H,2}^{n}\|_{a}^{2}+\cfrac{\tau}{2}\|u_{H}^{n}\|_{a}^{2}+\cfrac{\tau(1-\omega)}{2}\|u_{H,2}^{n}\|_{a}^{2}.
\end{align*}
Using the stability condition (\ref{eq:stab_cond}), 
we have 
\[
\cfrac{\gamma^{2}}{2}\sum_{i=1,2}\|u_{H,i}^{n+1}-u_{H,i}^{n}\|^{2}+\cfrac{\tau}{2}\|u_{H}^{n+1}\|_{a}^{2}\leq\cfrac{\gamma^{2}}{2}\sum_{i=1,2}\|u_{H,i}^{n}-u_{H,i}^{n-1}\|^{2}+\cfrac{\tau}{2}\|u_{H}^{n}\|_{a}^{2}.
\]
\end{proof}

We present a generalized version of the above theorem. We further assume that the space $V_{H,2}$ can be decomposed as
\begin{equation}
\label{eq:V2decomp}
V_{H,2} = \sum_{j=1}^{J} V_{H,2,j}.
\end{equation}
By using the strengthened Cauchy-Schwarz inequality, there is a constant $0<\beta_{j,m}<1$ such that
\begin{equation}
\beta_{j,m} := \sup_{v_j\in V_{H,2,j} , v_m\in V_{H,2,m}} \cfrac{(v_j,v_m)}{\|v_j\| \|v_m\|}, \quad \forall j,m=1,\cdots, J, \; j\ne m.
\end{equation}
Using this, we have for any $j,m$ with $j\ne m$,
\begin{equation}
\| v_j \|^2 + \|v_m\|^2 \leq \cfrac{1}{1-\beta^2_{j,m}} \| v_j + v_m\|^2.
\end{equation}
Let $\beta = \max \beta_{j,m}$. Then we have 
\begin{equation}
\label{eq:CS}
\sum_{j=1}^J \|v_j\|^2 \leq (\cfrac{1}{1-\beta^2})^\ell \|v\|^2, \quad v = \sum_{j=1}^J v_j
\end{equation}
where $\ell$ is the smallest integer greater than or equal to $\log_2 J$.

\begin{lemma}
\label{lem:overlap}
Assume that $V_{H,2}$ has the decomposition defined in (\ref{eq:V2decomp}). Then we have
\[
\sup_{v\in V_{H,2}}\cfrac{\|v\|^2_{a}}{\|v\|^2} \leq J^2(1-\beta^2)^{-\ell} \sup_{1\leq j\leq J} \sup_{v_j\in V_{H,2,j}}\cfrac{\|v_j\|^2_{a}}{\|v_j\|^2}
\]
where $\ell$ is the smallest integer greater than or equal to $\log_2 J$.
\end{lemma}
\begin{proof}
Let $v\in V_{H,2}$. Using the decomposition (\ref{eq:V2decomp}), 
\[
v = \sum^J_{j=1}v_j, \quad \text{ where }v_j\in V_{H,2,j}. 
\]
Since $\|v\|_a\leq \sum^J_{j=1} \|v_j\|_a$, we have 
\[
\|v\|^2_a \leq J^2\sum^J_{j=1} \|v_j\|^2_a.
\]
Since $\|v_j\|^2_a\leq \sup_{v_j\in V_{H,2,j}}\cfrac{\|v_j\|^2_{a}}{\|v_j\|^2} \|v_j\|^2$, we have 
\[
\|v\|^2_a \leq J^2\sup_{1\leq j\leq J}\sup_{v_j\in V_{H,2,j}}\cfrac{\|v_j\|^2_{a}}{\|v_j\|^2} \sum^J_{j=1} \|v_j\|^2.
\]
Using (\ref{eq:CS}), we then obtain
\[
\|v\|^2_a \leq J^2\sup_{1\leq j\leq J} \sup_{v_j\in V_{H,2,j}}\cfrac{\|v_j\|^2_{a}}{\|v_j\|^2} (1-\beta^2)^{-\ell}\|v\|^2.
\]

\end{proof}
\begin{theorem}
\label{thm33}
The partially explicit scheme (\ref{eq:sch1})-(\ref{eq:sch2})   is stable if 
\begin{equation}
\tau \sup_{v\in V_{H,2,j}} \cfrac{\|v\|_a^2}{\|v\|^2} \leq  J^{-2} \cfrac{(1 -\gamma^{2})}{2-\omega}(1-\beta^2)^\ell, \quad \forall j=1,2,\cdots, J,
\label{eq:stab_cond1}
\end{equation}
where $\ell$ is the smallest integer greater than or equal to $\log_2 J$.
Moreover, we have the following stability estiamte
\[
\cfrac{\gamma^{2}}{2}\sum_{i=1,2}\|u_{H,i}^{n+1}-u_{H,i}^{n}\|^{2}+\cfrac{\tau}{2}\|u_{H}^{n+1}\|_{a}^{2}\leq\cfrac{\gamma^{2}}{2}\sum_{i=1,2}\|u_{H,i}^{n}-u_{H,i}^{n-1}\|^{2}+\cfrac{\tau}{2}\|u_{H}^{n}\|_{a}^{2},\;\text{for }n\geq1.
\]
\end{theorem}

\begin{proof}
By Lemma \ref{lem:overlap}, we have
\[
\sup_{v\in V_{H,2}}\cfrac{\|v\|^2_{a}}{\|v\|^2} \leq J^2(1-\beta^2)^{-\ell} \sup_{1\leq j\leq J} \sup_{v_j\in V_{H,2,j}}\cfrac{\|v_j\|^2_{a}}{\|v_j\|^2}.
\]
Thus, if 
\[
\tau \sup_{v\in V_{H,2,j}} \cfrac{\|v\|_a^2}{\|v\|^2} \leq  J^{-2} \cfrac{(1 -\beta_0^{2})}{2-\omega}(1-\beta^2)^\ell, \quad \forall j=1,2,\cdots, J,
\]
we have 
\[
\tau \sup_{v\in V_{H,2}}\cfrac{\|v\|^2_{a}}{\|v\|^2} \leq \tau J^2(1-\beta^2)^{-\ell}\sup_{v_j\in V_{H,2,j}}\cfrac{\|v_j\|^2_{a}}{\|v_j\|^2}\leq \cfrac{1-\beta_0^{2}}{(2-\omega)}.
\]
The required result is obtained by using Theorem \ref{thm32}.
\end{proof}
Theorem \ref{thm33} is a generalization of Theorem \ref{thm32} in a sense
that it only requires a stability condition in each subspace.

\section{Spaces construction}

In this section, we will introduce one of the possible ways to construct
the spaces satisfying (\ref{eq:stab_cond}) or (\ref{eq:stab_cond1}).
We will show that the constrained energy minimization finite element space \cite{chung2018constraint}
is a good choice of $V_{H,1}$ since the CEM basis functions are constructed
such that they are almost orthogonal to a space $\tilde{V}$ which
can be easily defined. To obtain a $V_{H,2}$ satisfying the condition
(\ref{eq:stab_cond}) or (\ref{eq:stab_cond1}), one of the possible ways is using an eigenvalue
problem to construct the local basis function. Before, discussing
the construction of $V_{H,2}$, we will first introduce the 
CEM finite
element space $V_{H,1}$.

\subsection{The CEM-GMsFEM method}
\label{sec:cem}

In this section, we will discuss the CEM method \cite{chung2018constraint} for solving the problem
(\ref{eq:problem_weak}). The CEM method follows the framework of finite
element methods. We will construct the finite element space by solving
a constrained energy minimization problem. We let $\mathcal{T}_{H}$
be a coarse grid partition of $\Omega$ with $N_e$ elements.
For each coarse element $K_{i}\in\mathcal{T}_{H}$,
we consider a set of auxiliary basis functions $\{\psi_{j}^{(i)}\}_{j=1}^{L_{i}}\in V(K_{i})$ by solving
\begin{equation}
\label{eq:spectralCEM}
\int_{K_i} \kappa \nabla \psi_j^{(i)} \cdot \nabla v = \lambda_j^{(i)} s_i(\psi_j^{(i)},v), \quad \forall v\in V(K_i)
\end{equation}
and collecting the first $L_i$ eigenfunctions corresponding to the first $L_i$ smallest eigenvalues
with
\begin{equation}
\label{eq:eigenvalueproblem1}
s_{i}(u,v)=\int_{K_{i}}\tilde{\kappa}uv, 
\end{equation}
and $\tilde{\kappa}=\kappa H^{-2}$ or $\tilde{\kappa}=\kappa \sum_i|\nabla \chi_i|^2$, where $\{\chi_i\}$ is a set of partition of unity functions
corresponding to an overlapping partition of the domain. 

We define a projection operator $\Pi_{{i}}:L^{2}(K_{i})\mapsto V_{aux}^{(i)}\subset L^{2}(K_{i})$
such that 
\[
s_{i}(\Pi_{i}u,v)=s_{i}(u,v), \;\forall v\in V_{aux}^{(i)}:=\text{span}\{\psi_{j}^{(i)}:\;1\leq j\leq L_{i}\}.
\]
We next define a global projection operator by $\Pi:L^{2}(\Omega)\mapsto V_{aux}\subset L^{2}(\Omega)$
\[
s(\Pi u,v)=s(u,v),\;\forall v\in V_{aux}:=\sum_{i=1}^{N_{e}}V^{(i)}_{aux}
\]
where $s(u,v):=\sum_{i=1}^{N_{e}}s_{i}(u|_{K_{i}},v|_{K_{i}})$. 

For each auxiliary basis functions $\psi_{j}^{(i)}$, we can define
a global basis function $\phi_{glo,j}^{(i)}$ by 
\[
\phi_{glo,j}^{(i)}=\arg\min_{v\in V,\Pi v=\psi_{j}^{(i)}}\{a(v,v)\}.
\]
We can see that $\phi_{glo,j}^{(i)}$ will satisfy 
\begin{align*}
a(\phi_{glo,j}^{(i)},v)+s(\mu_{glo,j}^{(i)},v) & =0, \;\forall v\in V,\\
s(\phi_{glo,j}^{(i)},\nu) & =0,\;\forall\nu\in V_{aux},
\end{align*}
for some $\mu_{glo,j}^{(i)}\in V_{aux}$. To localize the basis function,
we will define the basis function $\phi_{j}^{(i)}\in V(K_{i}^{+})$
such that 
\begin{align*}
a(\phi_{j}^{(i)},v)+s(\mu_{j}^{(i)},v) & =0,\;\forall v\in V(K_{i}^{+}), \\
s(\phi_{j}^{(i)},\nu) & =s(\psi_{j}^{(i)},\nu),\;\forall\nu\in V_{aux}(K_{i}^{+}), 
\end{align*}
where $K_{i}^{+}$ is an oversampling domain of $K_{i}$ obtained by enlarging $K_i$ by a few coarse grid layers. 

We then define the spaces $V_{glo}$ and $V_{cem}$ as 
\begin{align}
V_{glo} & :=\text{span}\{\phi_{glo,j}^{(i)}:\;1\leq i\leq N_{e},1\leq j\leq L_{i}\},\label{eq:glospace} \\
V_{cem} & :=\text{span}\{\phi_{j}^{(i)}:\;1\leq i\leq N_{e},1\leq j\leq L_{i}\}. \label{eq:cemspace}
\end{align}
The global solution $u_{glo}\in V_{glo}$ and the CEM solution $u_{cem} \in V_{cem}$ are respectively 
defined as 
\begin{align*}
(u_{glo,t},v) & +a(u_{glo},v)=0, \;\forall v\in V_{glo},\\
(u_{cem,t},v) & +a(u_{cem},v)=0, \;\forall v\in V_{cem}
\end{align*}
where $u_{glo,t}$ and $u_{cem,t}$ are the time derivatives of $u_{glo}$ and $u_{cem}$ respectively. 
We note that $u_{cem}$ is a multiscale approximation of $u$,
and its convergence is analyzed in \cite{li2019constraint}. 

We remark that the $V_{glo}$ is $a-$orthogonal to a space $\tilde{V}:=\{v\in V:\;\Pi(v)=0\}$.
We also know that $V_{cem}$ is closed to $V_{glo}$ and therefore
it is almost orthogonal to $\tilde{V}$. Thus, we can choice $V_{cem}$
to be $V_{H,1}$ and construct a space $V_{H,2}$ in $\tilde{V}$. 

\subsection{Construction of $V_{H,2}$}

In this section, we present two choices for the space $V_{H,2}$ which will give an explicit stability condition based on (\ref{eq:stab_cond}) or (\ref{eq:stab_cond1}), 
and these choices are motivated by reducing errors (see the Appendix). 
Recall that $V= H^1_0(\Omega)$. For any set $S$, we let $V(S) = H^1(S)$ and $V_0(S) = H^1_0(S)$.

\subsubsection{First choice}
\label{sec:choice1}

We will define basis functions for each coarse neighborhood $\omega_i$,
which is the union of all coarse elements having the $i$-th coarse grid node. 
For each coarse neighborhood $\omega_{i}$, we consider the following
eigenvalue problem: find $(\xi_{j}^{(i)},\gamma_{j}^{(i)})\in ( V_0(\omega_i) \cap \tilde{V}) \times\mathbb{R}$,
\begin{align}
\label{eq:eigenvalueproblem_case2}
\int_{\omega_{i}}\kappa\nabla\xi_{j}^{(i)}\cdot\nabla v & = \cfrac{\gamma_{j}^{(i)}}{H^2}\int_{\omega_{i}}\xi_{j}^{(i)}v, \;\forall v\in V_0(\omega_i) \cap \tilde{V}.
\end{align}
We arrange the eigenvalues by $\gamma_1^{(i)}  \leq \gamma_2^{(i)} \geq \cdots$.
In order to obtain a reduction in error, we will select the first few $J_i$ dominant eigenfunctions corresponding to smallest eigenvalues of (\ref{eq:eigenvalueproblem_case2}). 
We define
\begin{equation*}
V_{H,2} = \text{span} \{ \xi_j^{(i)} \; | \; \forall \omega_i, \forall 1\leq j \leq J_i\}.
\end{equation*}

We assume that the domain $\Omega$ is a square, and that the coarse grid $\mathcal{T}_H$ is a regular mesh. Then, we can
partition the set of all coarse neighborhoods $\omega_i$ into $4$ subsets, such that each subset contains disjoint coarse neighborhoods, see \cite{chung2015residual}
for more details. Based on this, we can subdivide $V_{H,2}$ into $4$ spaces $V_{H,2,j}$, $j=1,2,3,4$. 
Notice that, for each $j=1,2,3,4$, we have
\begin{equation*}
\| v\|_a^2 \leq (\max_{i} \gamma_{J_i+1}^{(i)}) H^{-2} \|v\|^2, \quad \forall v\in V_{H,2,j}.
\end{equation*}
Hence, an explicit form of the stability condition (\ref{eq:stab_cond1}) is given by
\begin{equation}
\tau  \leq (16)^{-1} (\max_i \gamma_{J_i+1}^{(i)})^{-1}  (1 -\gamma^{2})(1-\beta^2)^2 (2-\omega)^{-1} H^2.
\end{equation}
We remark that some motivation of this choice of $V_{H,2}$ is discussed in the Appendix. 
We observe that there is a tradeoff in using the number of functions in $V_{H,2}$, namely, more functions in $V_{H,2}$ will
lead to more severe stability condition. 

\subsubsection{Second choice}

The second choice of $V_{H,2}$ is based on the CEM type finite
element space. For
each coarse element $K_{i}$, we will solve an eigenvalue problem to obtain the auxiliary
basis. More precisely, we find eigenpairs $(\xi_{j}^{(i)},\gamma_{j}^{(i)})\in(V(K_{i})\cap\tilde{V})\times\mathbb{R}$ by solving
\begin{align}
\label{eq:spectralCEM2}
\int_{K_{i}}\kappa\nabla\xi_{j}^{(i)}\cdot\nabla v & =\gamma_{j}^{(i)}\int_{K_{i}}\xi_{j}^{(i)}v, \;\ \forall v\in V(K_{i})\cap\tilde{V}.
\end{align}
For each $K_i$, we choose the first few $J_i$ eigenfunctions corresponding to the smallest $J_i$ eigenvalues. 
The resulting space is called $V_{aux,2}$.
For each auxiliary basis function $\xi_j^{(i)}$, we define the global basis function $\zeta_{glo,j}^{(i)} \in V$ such
that $\mu_{glo,j}^{(i)} \in V_{aux,1}$, $ \mu_{glo,j}^{(i),2} \in V_{aux,2}$ and 
\begin{align}
a(\zeta_{glo,j}^{(i)},v)+s(\mu_{glo,j}^{(i),1},v)+ ( \mu_{glo,j}^{(i),2},v) & =0, \;\forall v\in V, \label{eq:v2a} \\
s(\zeta_{glo,j}^{(i)},\nu) & =0, \;\forall\nu\in V_{aux,1}, \label{eq:v2b} \\
(\zeta_{glo,j}^{(i)},\nu) & =( \xi_{j}^{(i)},\nu), \;\forall\nu\in V_{aux,2}. \label{eq:v2c}
\end{align}
where we use the notation $V_{aux,1}$ to denote the space $V_{aux}$ defined in Section \ref{sec:cem}.
We recall that the basis function $\phi_{glo,j}^{(i)}$ in $V_{glo}$ constructed in Section \ref{sec:cem}
satisfies
\begin{align*}
a(\phi_{glo,j}^{(i)},v)+s(\mu_{glo,j}^{(i)},v) & =0, \;\forall v\in V, \\
s(\phi_{glo,j}^{(i)},\nu) & =s(\psi_{j}^{(i)},\nu),\;\forall\nu\in V_{aux,1}
\end{align*}
where $\mu_{glo,j}^{(i)} \in V_{aux,1}$. 
Thus, taking $v = \zeta_{glo,l}^{(k)}$ in the above system, 
we have 
\begin{align*}
a(\phi_{glo,j}^{(i)},\zeta_{glo,l}^{(k)}) & =-s(\mu_{glo,j}^{(i)},\zeta_{glo,l}^{(k)})=0,\quad \;\forall i,j,k,l.
\end{align*}
We define $V_{glo,2}=\text{span}\{\zeta_{glo,j}^{(i)}|\;j\leq J_i\}$. 
This is our choice of $V_{H,2}$, that is, we take $V_{H,2} = V_{glo,2}$.

Now, we will derive a more explicit stability condition based on (\ref{eq:stab_cond}). 
To do so, we define a projection operator $\tilde{\Pi}:L^{2}(\Omega)\rightarrow V_{aux,1}+V_{aux,2}$ by
\[
\tilde{\Pi}(v)=v_{aux,1}+v_{aux,2}
\]
where 
\[
v_{aux,1}=\Pi(v)
\]
and 
\[
(v_{aux,2},w)=(v-\Pi(v),w), \;\;\forall w\in V_{aux,2}.
\]
We remark that for any $v\in L^2(\Omega)$, we have
\begin{equation}
\label{eq:projv2}
(\tilde{\Pi}(v),w)=(v_{aux,1}+v_{aux,2},w)=(v_{aux,1}+v-v_{aux,1},w)=(v,w), \;\;\forall w\in V_{aux,2}.
\end{equation}
We assume that,
for each $v\in V_{aux,2}(K_{i})$, there exist a $P(v)\in V_{0}(K_{i})$
such that 
\begin{equation}
\label{eq:assumption}
\tilde{\Pi} P(v)=v,\quad \|P(v)\|_{a}\leq C_{1}H^{-1}\|v\|,
\end{equation}
where $C_1$ is independent of the contrast.

For any $\tilde{v}\in V_{glo,2}$, we have $s(\tilde{v},\nu) = 0 $ for all $\nu \in V_{aux,1}$.
Thus, using (\ref{eq:v2a}), (\ref{eq:v2b}) and (\ref{eq:projv2}), there are $\mu^{(1)} \in V_{aux,1}$ and $\mu^{(2)} \in V_{aux,2}$ such that 
\begin{align*}
a(\tilde{v},v)+s(\mu^{(1)},v)+\int_{\Omega}\mu^{(2)}v & =0, \;\;\forall v\in V, \\
s(\tilde{v},\nu) & =0,\;\;\forall\nu\in V_{aux,1}, \\
  \int_{\Omega}\tilde{v}\nu & =  \int_{\Omega}\tilde{\Pi}(\tilde{v})\nu,\;\;\forall\nu\in V_{aux,2}.
\end{align*}
Taking $v = \tilde{v}$, we have 
\begin{align*}
a(\tilde{v},\tilde{v}) & =-\int_{\Omega}\mu^{(2)}\tilde{v}\leq\|\mu^{(2)}\| \, \|\tilde{\Pi}(\tilde{v})\|.
\end{align*}
On the other hand, by (\ref{eq:assumption}), there is $P(\mu^{(2)})$ such that $\tilde{\Pi} P\mu^{(2)} = \mu^{(2)}$. 
By definition of $\tilde{\Pi}$, we have 
\[
\tilde{\Pi}(P\mu^{(2)}) = \Pi(P\mu^{(2)})+ v_{aux,2}
\]
where $\Pi(P\mu^{(2)})\in V_{aux,1}$ and $v_{aux,2}\in V_{aux,2}$. 
Note that $V_{aux,1} + V_{aux,2}$ is a direct sum due to the $s$-orthogonality of the spaces $V_{aux,1}$ and $V_{aux,2}$.
So, using the assumption $\mu^{(2)}=\tilde{\Pi}(P\mu^{(2)})$, we have 
\begin{align*}
\Pi(P\mu^{(2)}) & = 0\\
\mu^{(2)} & = v_{aux,2}
\end{align*}
which implies 
\begin{align*}
s(P(\mu^{(2)}), \nu) & =s(\Pi(P(\mu^{(2)})), \nu) = 0, \;\;\forall \nu \in V_{aux,1},\\
(P(\mu^{(2)}), \nu) & =(\tilde{\Pi}P(\mu^{(2)}), \nu)  = (\mu^{(2)},\nu),\;\; \forall \nu \in V_{aux,2}.
\end{align*}
Thus, we have
\begin{align*}
\|\mu^{(2)}\|^{2} & =\int_{\Omega}\mu^{(2)} \tilde{\Pi} P(\mu^{(2)})= \int_{\Omega}\mu^{(2)}  P(\mu^{(2)}) \\
 & = s(\mu^{(1)},P(\mu^{(2)}))+\int_{\Omega}\mu^{(2)}P(\mu^{(2)}) = -a(\tilde{v},P(\mu^{(2)}))  \\
 & \leq C_{1}H^{-1}\|\mu^{(2)}\|_{L^{2}}\|\tilde{v}\|_{a},
\end{align*}
which implies
\[
a(\tilde{v},\tilde{v})\leq C_{1}^{2}H^{-2}\|\tilde{\Pi}(\tilde{v})\|^{2}\leq C_{1}^{2}H^{-2}\|\tilde{v}\|^{2}, \quad \forall \tilde{v} \in V_{H,2}.
\]
Thus, the stability condition (\ref{eq:stab_cond}) becomes
\[
\tau \leq C_1^{-2} H^2 \cfrac{1-\gamma^2}{2-\omega}.
\]
Therefore, we have $V_{glo},\;V_{glo,2}$ satisfy the required condition
for our time-splitting method. To localize the space $V_{glo,2}$, we can use similar ideas
as in the CEM method.


We can see one of the major differences of the eigenvalue problem (\ref{eq:spectralCEM}) and the eigenvalue problem (\ref{eq:spectralCEM2}) is that the eigenvalue problem (\ref{eq:spectralCEM2}) prefers to select some basis functions representing solution in regions with low permeability. We can show that the $L_2$-norm of the basis functions in the high permeability regions is inversely proportional to the contrast of the permeability. Hence, it is reasonable to assume the auxiliary basis functions satisfy the condition (\ref{eq:assumption}). We will present numerical results for the constant $C_1$ in the next section and investigate some
simplified cases analytically.

Next, we will discuss a simple case to demonstrate how the above 
condition  can be satisfied with a contrast independent $C_1$. 
We consider 
\[
\kappa=\begin{cases}
1 & x\in\Omega\backslash\Omega_{\kappa}\\
\kappa_{\max} & x\in\Omega_{\kappa}
\end{cases}
\]
for some $\Omega_{\kappa}\subset\Omega$ where $\kappa_{\max}>>1$.

In this case, we can consider the auxiliary space $V_{aux,1}$with
the following eigenvalue problem: 
\begin{align*}
\int_{K_{i}}\kappa\nabla\xi_{j}^{(i)}\cdot\nabla v & =\gamma_{j}^{(i)}s(\xi_{j}^{(i)},v)\\
 & =\gamma_{j}^{(i)}\int_{K_{i}}\tilde{\kappa}\xi_{j}^{(i)}v\;\forall v\in V_{\kappa}(K_{i})\cap\tilde{V},
\end{align*}
where $\tilde{\kappa}=H^{-2}\kappa$ and $V_{\kappa}(K_{i})=\{v\in V(K_{i})|\;v|_{\Omega_{\kappa}}=0\}$.

We then consider the auxiliary space $V_{aux,2}$ with the following
eigenvalue problem: 
\begin{align*}
\int_{\omega_{i}}\kappa\nabla\xi_{j}^{(i)}\cdot\nabla v & =\gamma_{j}^{(i)}H^{-2}\int_{\omega_{i}}\xi_{j}^{(i)}v\;\forall v\in V_{\kappa}(K_{i})\cap\tilde{V},
\end{align*}
where $V_{\kappa}(K_{i})=\{v\in V(K_{i})|\;v|_{\Omega_{\kappa}}=0\}$.

For each coarse element $K_{i}$, we define a bubble function $B_{i}\in C_{0}^{\infty}(K_{i})$
such that 
\[
\|\nabla B_{i}\|_{L^{\infty}}\leq DH^{-1}\;\text{and }\|B_{i}\|_{L^{\infty}}=1\;\forall i.
\]
We then consider a function $B\in V$ such that $B|_{K_{i}}=B_{i}$
for any $i$. 

We define two constants $\gamma_{1}<1$ and $\gamma_{2}$ such that
\[
\gamma_{1}:=\sup_{w_{1}\in V_{aux,1},w_{2}\in V_{aux,2}}\cfrac{\int_{\Omega}Bw_{1}w_{2}}{\|B^{\frac{1}{2}}w_{1}\|_{L^2}\|B^{\frac{1}{2}}w_{2}\|_{L^2}}
\]
\[
\gamma_{2}:=\sup_{w_{1}\in V_{aux,1},w_{2}\in V_{aux,2}}\cfrac{\|w_{2}\|_{L^2}}{\|B^{\frac{1}{2}}w_{2}\|_{L^2}}.
\]

\begin{lemma}
\label{lem:simple}
If we consider a simplified case described above, there exist a $C>0$ such that 
\[
\|v\|_{a}\leq CH^{-1}\|v\|_{L^{2}}\;\forall v\in V_{glo,2}.
\]
\end{lemma}

\begin{proof}
Given a $w\in V_{glo,2}$, by the definition of $V_{glo,2}$, we have
\begin{align*}
a(w,v)+s(\mu_{w}^{1},v)+H^{-2}(\mu_{w}^{2},v) & =0\;\forall v\in V\\
s(w,\nu) & =0\;\forall\nu\in V_{aux,1}\\
(w,\nu) & =(\tilde{\Pi}(w),\nu)\;\forall\nu\in V_{aux,2}.
\end{align*}
Thus, we have 
\begin{align*}
a(w,w) & =-s(\mu_{w}^{1},w)-H^{-2}(\mu_{w}^{2},w)\\
 & \leq H^{-2}\|\tilde{\Pi}(w)\|_{L^2}\|\mu_{w}^{2}\|_{L^2}.
\end{align*}
We have
\begin{align*}
H^{-2}(\mu_{w}^{2},B\mu_{w}^{2}) & =-a(w,B\mu_{w}^{2})-s(\mu_{w}^{1},B\mu_{w}^{2})\\
 & \leq\|w\|_{a}\|B\mu_{w}^{2}\|_{a}+\gamma_{1}H^{-2}\|B^{\frac{1}{2}}\mu_{w}^{2}\|_{L^2}\|B^{\frac{1}{2}}\mu_{w}^{1}\|_{L^2}\\
 & \leq\|w\|_{a}\|B\mu_{w}^{2}\|_{a}+\cfrac{\gamma_{1}}{2}H^{-2}\Big(\|B^{\frac{1}{2}}\mu_{w}^{2}\|_{L^2}^{2}+\|B^{\frac{1}{2}}\mu_{w}^{1}\|_{s}^{2}\Big)
\end{align*}
\begin{align*}
s(\mu_{w}^{1},B\mu_{w}^{1}) & =-a(w,B\mu_{w}^{1})-H^{-2}(\mu_{w}^{2},B\mu_{w}^{1})\\
 & \leq\|w\|_{a}\|B\mu_{w}^{2}\|_{a}+\cfrac{\gamma_{1}}{2}H^{-2}\Big(\|B^{\frac{1}{2}}\mu_{w}^{2}\|_{L^2}^{2}+\|B^{\frac{1}{2}}\mu_{w}^{1}\|_{s}^{2}\Big).
\end{align*}
Consequently,
\[
(1-\gamma_{1})\Big(H^{-2}\|B^{\frac{1}{2}}\mu_{w}^{2}\|_{L^2}^{2}+\|B^{\frac{1}{2}}\mu_{w}^{1}\|_{s}^{2}\Big)\leq\|w\|_{a}(\|B\mu_{w}^{2}\|_{a}+\|B\mu_{w}^{1}\|_{a}).
\]
We have
\begin{align*}
\|B\mu_{w}^{i}\|_{a}^{2} & =\int\kappa|\nabla(B\mu_{w}^{k})|^{2}\\
 & \leq2\Big(\int\kappa|\nabla B|^{2}|\mu_{w}^{k}|^{2}+\int\kappa|B|^{2}|\nabla\mu_{w}^{k}|^{2}\Big)\;\text{for }k=1,2.
\end{align*}
Thus, we have 
\begin{align*}
\|B\mu_{w}^{1}\|_{a}^{2} & \leq D\|\mu_{w}^{1}\|_{s}^{2}+\int\kappa|\nabla\mu_{w}^{1}|^{2}\\
 & \leq D(1+\max_{i}\{\lambda_{L_{i}}^{(i)}\})\|\mu_{w}^{1}\|_{s}^{2}
\end{align*}
and 
\begin{align*}
\|B\mu_{w}^{2}\|_{a}^{2} & \leq DH^{-2}\|\mu_{w}^{2}\|_{L^2}^{2}+\int\kappa|\nabla\mu_{w}^{2}|^{2}\\
 & \leq D(1+\max_{i}\{\gamma_{J_{i}}^{(i)}\})\|\mu_{w}^{2}\|_{s}^{2}.
\end{align*}
We then have 
\[
\Big(H^{-2}\|B^{\frac{1}{2}}\mu_{w}^{2}\|_{L^2}^{2}+\|B^{\frac{1}{2}}\mu_{w}^{1}\|_{s}^{2}\Big)^{\frac{1}{2}}\leq C(1+M)(1-\gamma_{1})^{-1}\|w\|_{a},
\]
where $E=\max_{i}\{\lambda_{L_{i}}^{(i)},\gamma_{J_{i}}^{(i)}\}$.
Hence, we have 
\begin{align*}
H^{-1}\|\mu_{w}^{2}\|_{L^2} & \leq\gamma_{2}H^{-1}\|B^{\frac{1}{2}}\mu_{w}^{2}\|_{L^2}\\
 & \leq CD(1+M)(1-\gamma_{1})^{-1}\|w\|_{a}
\end{align*}
and 
\begin{align*}
\|w\|_{a} & \leq CD(1+E)(1-\gamma_{1})^{-1}H^{-1}\|\tilde{\Pi}(w)\|_{L^2}\\
 & \leq CD(1+E)(1-\gamma_{1})^{-1}H^{-1}\|w\|_{L^2}.
\end{align*}

\end{proof}

Since the CEM basis functions exponentially converge to the global basis functions, we conclude that the space $V_{H,2}$ also satisfies Lemma \ref{lem:simple} if we use a large enough oversampling domain.

\section{Numerical Result}

In this section, we present representative numerical results
that show that proposed approaches can select time step independent
of the contrast and predict an accurate approximation for the solution.
 We consider the following 
parameters for the mesh sizes, and time steps
\[
H=1/10,\;h=1/100,\;dt=10^{-4},\;T=0.05.
\]
Here, $H$ is the coarse mesh size, $h$ is the fine mesh size, 
$dt$ is the fine time step, and $T$ is the final time step.
The conductivity 
fields and forcing terms are chosen differently for examples and
described in each part.

In our first numerical example, we choose a smooth source term.
In this case, CEM-GMsFEM without additional basis functions provide
results similar to those CEM-GMsFEM with additional basis functions
that are treated explicitly in our method. 
In this paper, we do 
not dwell on accuracy issues related the use of additional basis functions
in CEM-GMsFEM (that are treated explicitly). These basis functions
are needed in many cases to capture dynamics effects, in wave equations,
and so on. We will discuss this in our future works.

\subsection*{Numerical Example 1}

The medium parameter $\kappa$, the reference solution at final time
$u_{ref}$, and the source term $f$ are shown in Figure \ref{fig:perms3}.
As we see, the permeability field is heterogeneous with high contrast
streaks. Due to smooth source term, the solution's features in high conductivity
field regions are smeared. 

In Figure \ref{fig:results3.1}, we depict the error in $L_2$ 
and in energy norm
that correspond to three methods. The blue curve denotes the error
due to CEM without additional basis functions. Because of smooth source
term and problem setup, this method provides an error that is comparable
to the error when we consider additional basis functions. The additional
degrees of freedom treated both implicitly (red curve) and explicitly
(yellow curve). As we see that these two curves coincide. This indicates
that the time stepping that is chosen independent of contrast provides
as accurate solution as full backward Euler for our proposed partial explicit method. Consequently, this backs up
our discussions. In  Figure \ref{fig:results3.1}, we consider 
$V_{2,H}$, which is the first type, and in  Figure \ref{fig:results3.2},
we consider the case with the second type $V_{2,H}$. The results are similar,
which show that both spaces provide a robust partial explicit discretization.

\begin{figure}[H]
\centering

\includegraphics[scale=0.35]{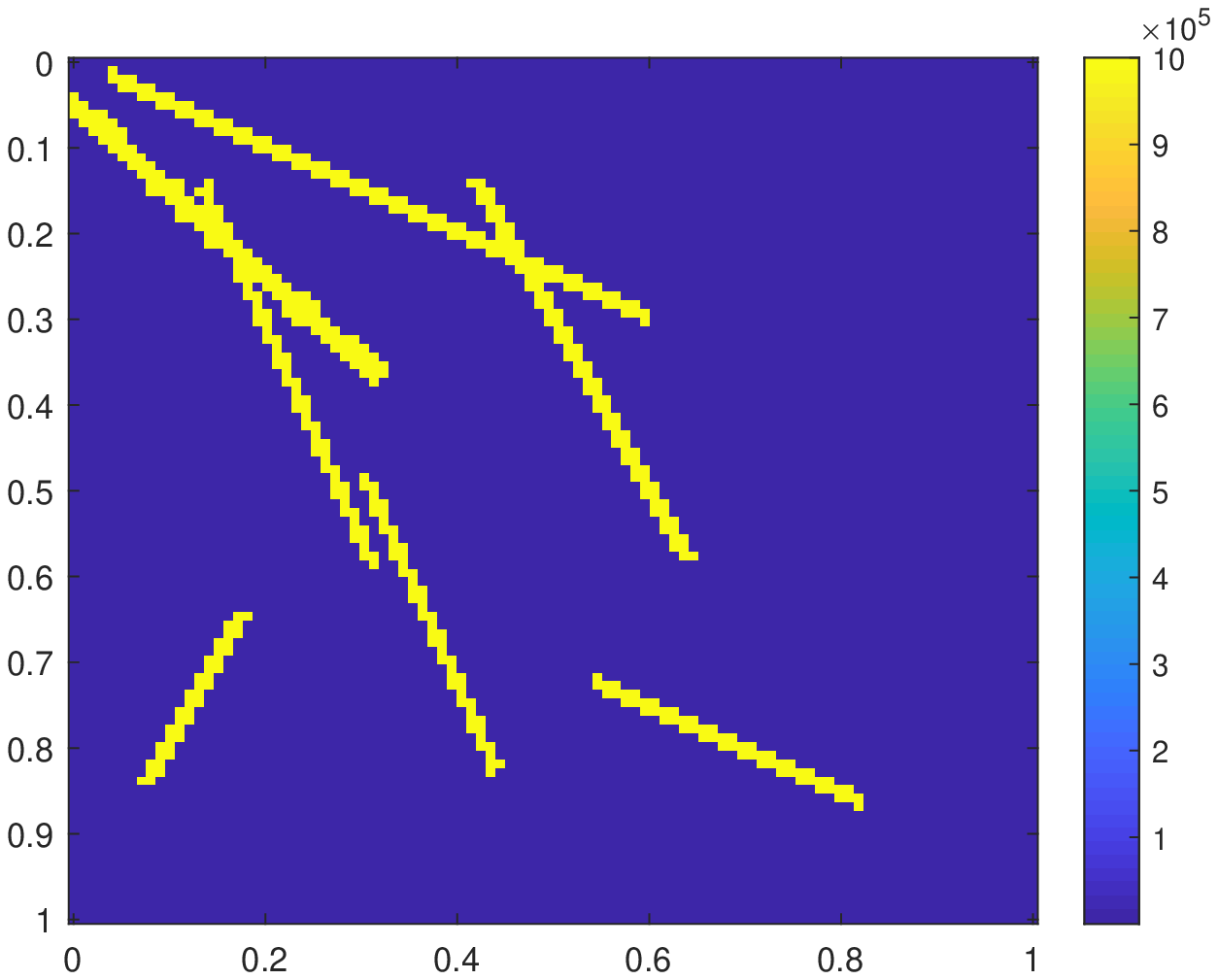} 
\includegraphics[scale=0.35]{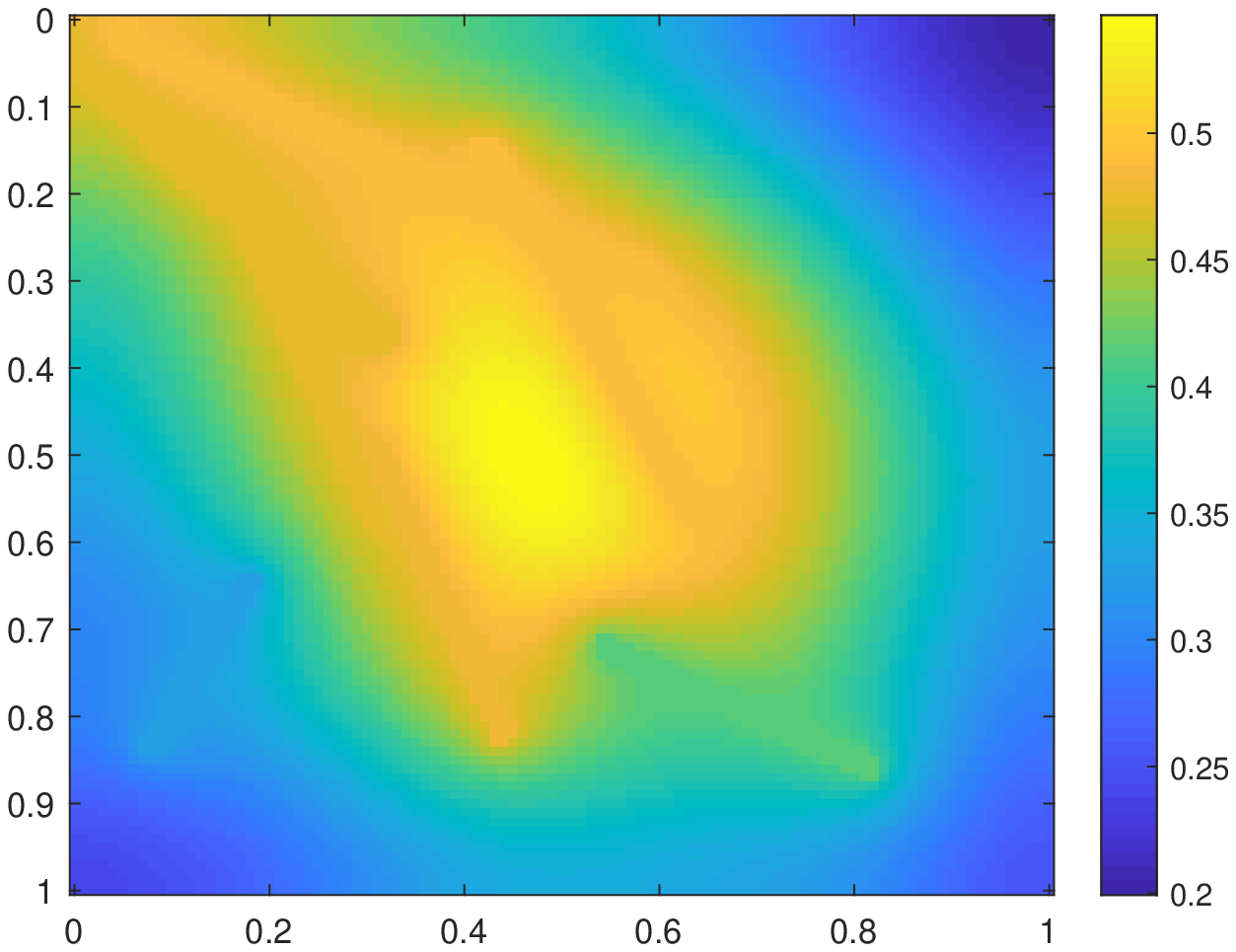}
\includegraphics[scale=0.35]{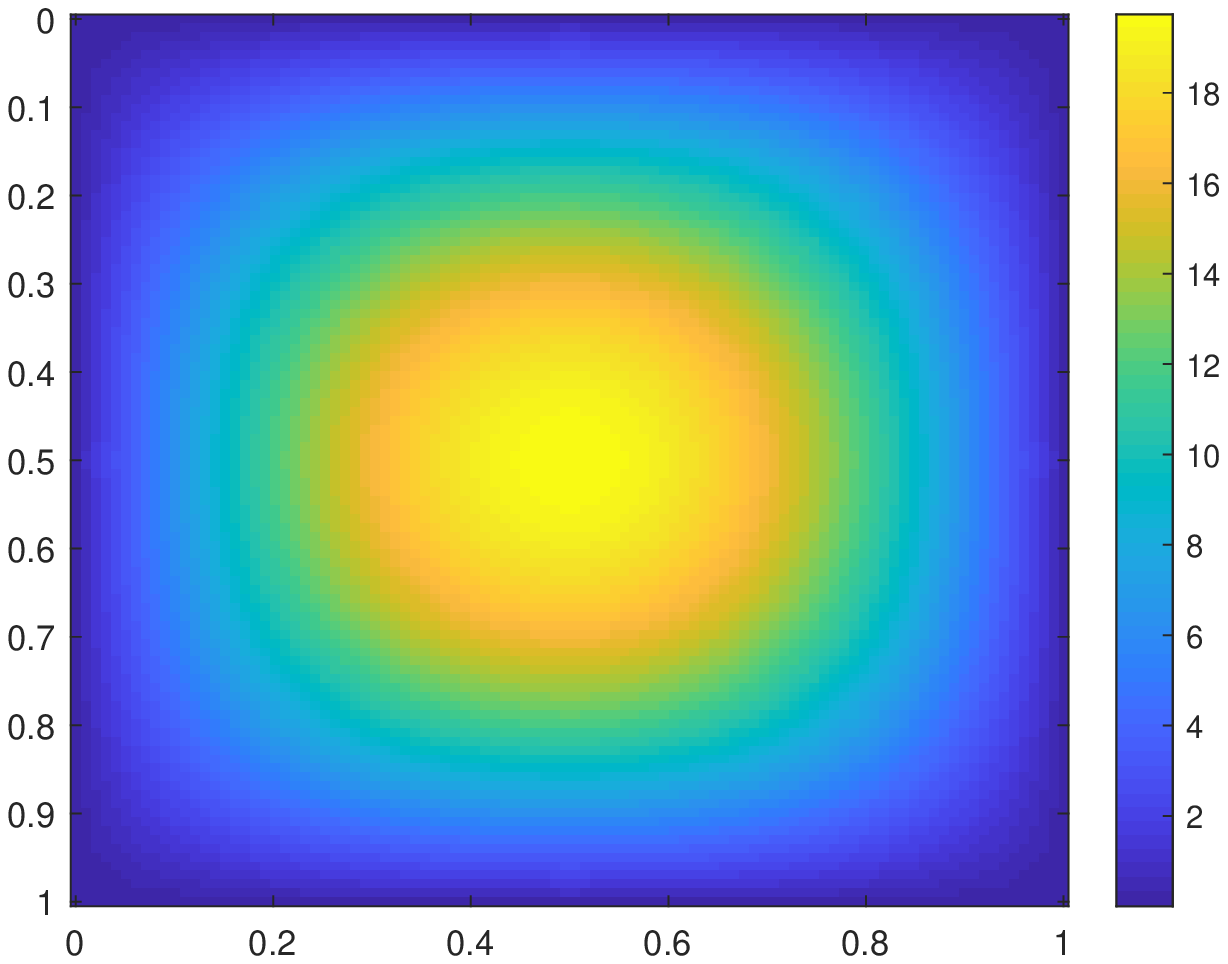}
\caption{Left: $\kappa$. Middle: reference solution at the final time. Right: $f$.}
\label{fig:perms3}
\end{figure}

\begin{figure}[H]
\centering

\includegraphics[scale=0.4]{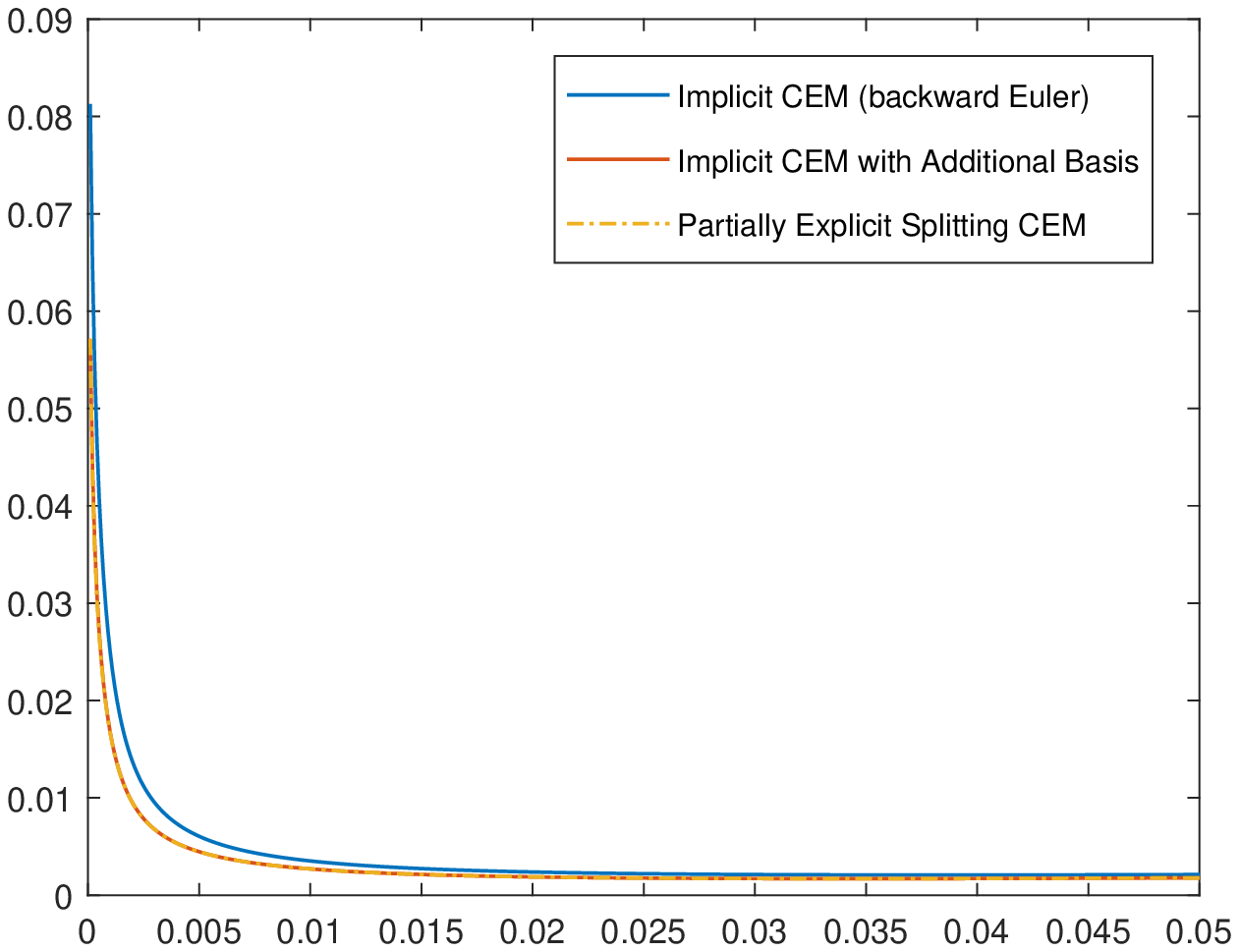} \includegraphics[scale=0.4]{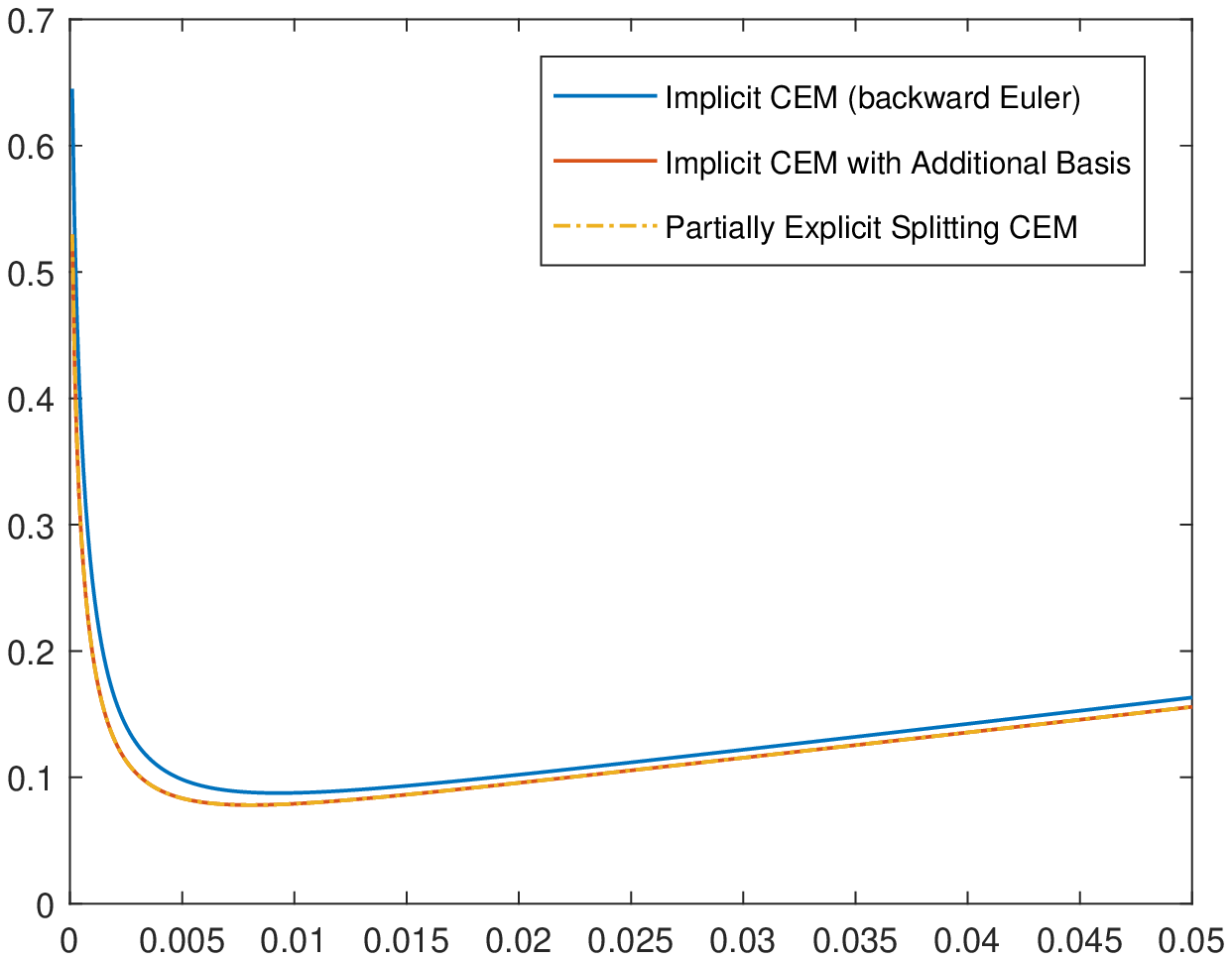}

\caption{Example 1. First type of $V_{2,H}$ (CEM Dof: $300$, $V_{2,H}$ Dof:
$243$). Left: $L_{2}$ error. Right: Energy error. Along $x$-axis is time, along $y$-axis is the relative error.}
\label{fig:results3.1}
\end{figure}

\begin{figure}[H]
\centering

\includegraphics[scale=0.4]{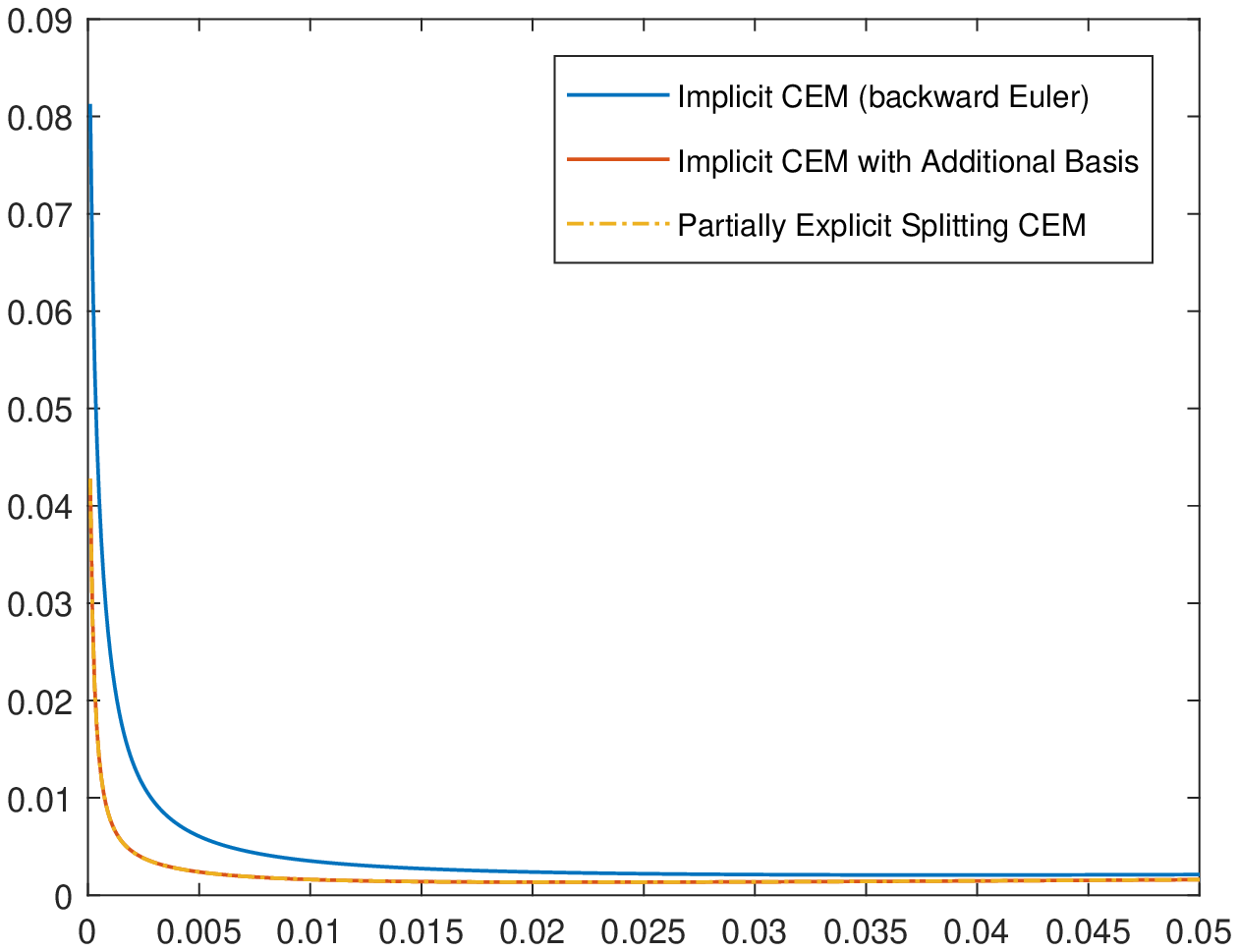} \includegraphics[scale=0.4]{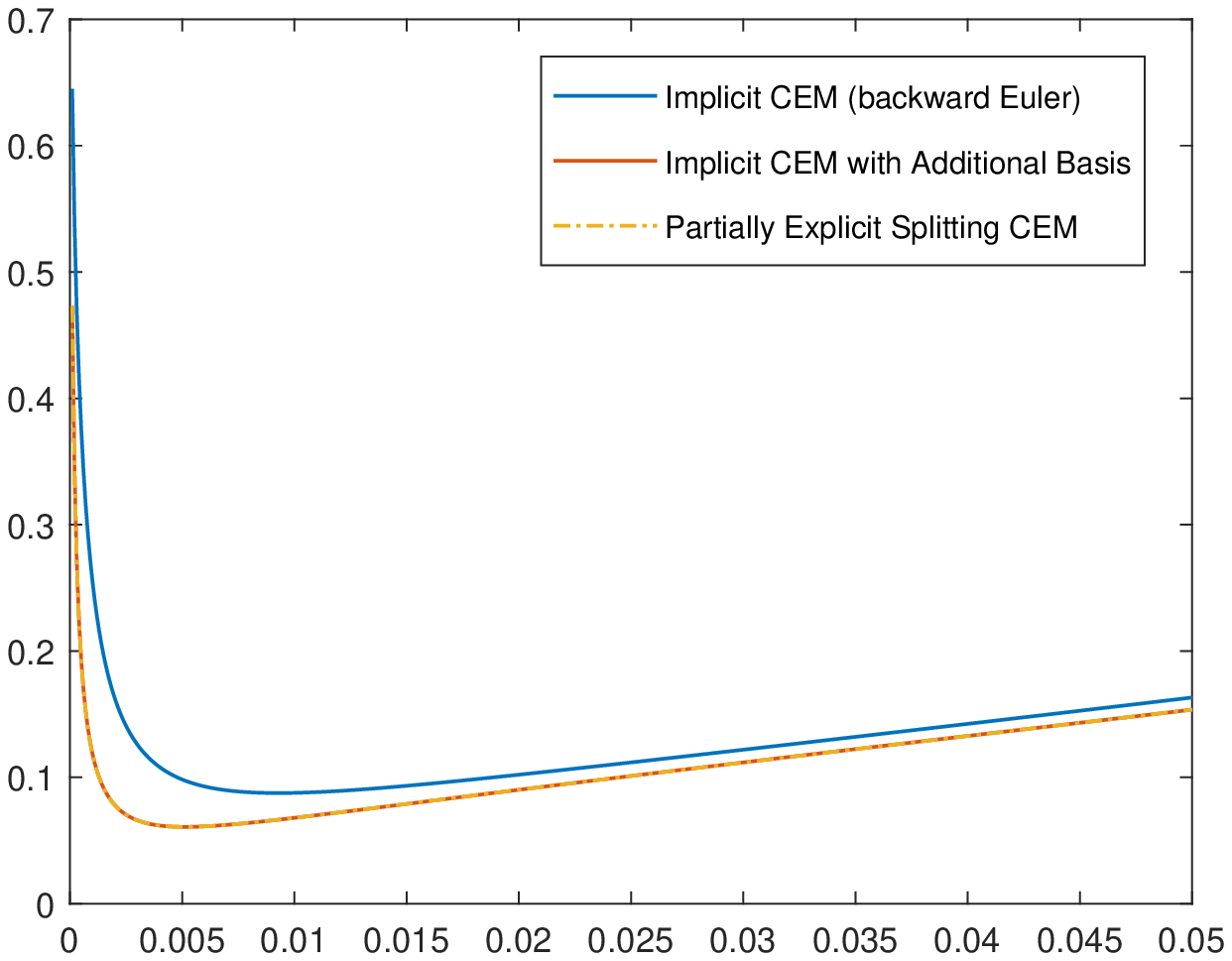}

\caption{Example 1. Second type of $V_{2,H}$ (CEM Dof: $300$, $V_{2,H}$
Dof: $300$). Left: $L_{2}$ error. Right: Energy error.  Along $x$-axis is time, along $y$-axis is the relative error.}
\label{fig:results3.2}
\end{figure}

\subsection*{Numerical Example 2}

The medium parameter $\kappa$, the reference solution at final time
$u_{ref}$, and the source term $f$ are shown in 
Figure \ref{fig:perms1}. We note that we intentionally choose
a singular source term so that CEM with additional basis functions
can give a substantial improvement as original multiscale CEM basis
functions do not take into account singular source term. In this case,
CEM-GMsFEM errors are large.
First, we numerically compute the constant $C_1$ from 
(\ref{eq:assumption}) that is
assumed to be independent of the contrast. 
 The result is shown in Table \ref{tab:constant1} . As
we see from this table that as we increase the contrast, this constant
remains constant, which asserts that our assumption is true. 
Next, present numerical results. 

In Figure \ref{fig:results1.1}, we present numerical results 
(the errors due to discretization),
when
$V_{2,H}$ is chosen as the first type. Again, we note that
because of singular source term, CEM with additional basis functions
will provide a visible improvement over CEM without using additional basis
functions. This is clear from the figure as we compare blue line
(CEM without additional basis functions) and other lines (which coincide)
that indicate results obtained using CEM with additional basis functions.
The two graphs that coincide correspond results using backward Euler
and partially explicit CEM-GMsFEM method with additional basis functions.
As we show that the errors are almost the same and thus, one can use 
our proposed approach with the time step independent of the contrast and
with partial explicit strategy. In Figure \ref{fig:results1.2},
we present results using  $V_{2,H}$ as the second type. As we see,
our results confirm that this space also provides numerical accuracy
as in the first type $V_{2,H}$.

\begin{table}[H]
\begin{tabular}{|c|c|c|c|c|c|}
\hline 
$\cfrac{\max\{\kappa\}}{\min\{\kappa\}}$ & $10^{5}$ & $10^{6}$ & $10^{7}$ & $10^{8}$ & $10^{9}$\tabularnewline
\hline 
$\sup_{v\in V_{1,H}} \mathcal{G}(v)$ & $4.11\times10^{5}$ & $4.11\times10^{6}$ & $4.11\times10^{7}$ & $4.11\times10^{8}$ & $4.11\times10^{9}$\tabularnewline
\hline 
First type of $V_{2,H}$. $\sup_{v\in V_{2,H}}\mathcal{G}(v)$ & $1.75\times10^{2}$ & $1.75\times10^{2}$ & $1.75\times10^{2}$ & $1.75\times10^{2}$ & $1.75\times10^{2}$\tabularnewline
\hline 
Second type of $V_{2,H}$. $\sup_{v\in V_{2,H}} \mathcal{G}(v)$ & $1.40\times10^{2}$ & $1.40\times10^{2}$ & $1.40\times10^{2}$ & $1.40\times10^{2}$ & $1.40\times10^{2}$\tabularnewline
\hline 
\label{tab:constant1}
\end{tabular}

\caption{Example 2. $\sup\cfrac{\|v\|_{a}^{2}}{\|v\|_{L^{2}}^{2}}$ for different
$\cfrac{\max\{\kappa\}}{\min\{\kappa\}}$. Here, we denote $\mathcal{G}(v)=\cfrac{\|v\|_{a}^{2}}{H^{-2}\|v\|_{L^{2}}^{2}}$.}
\end{table}

\begin{figure}[H]
\centering

\includegraphics[scale=0.35]{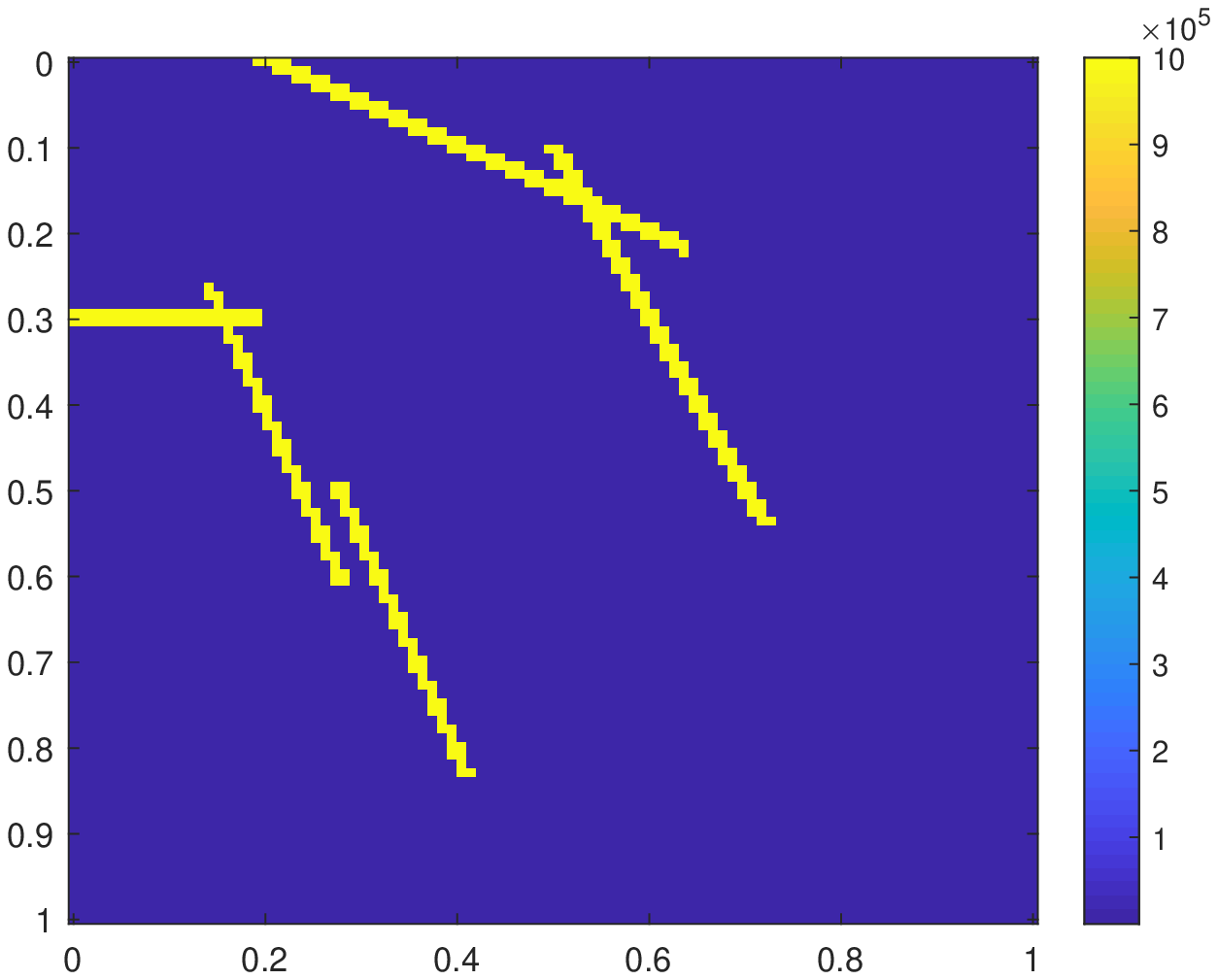} 
\includegraphics[scale=0.35]{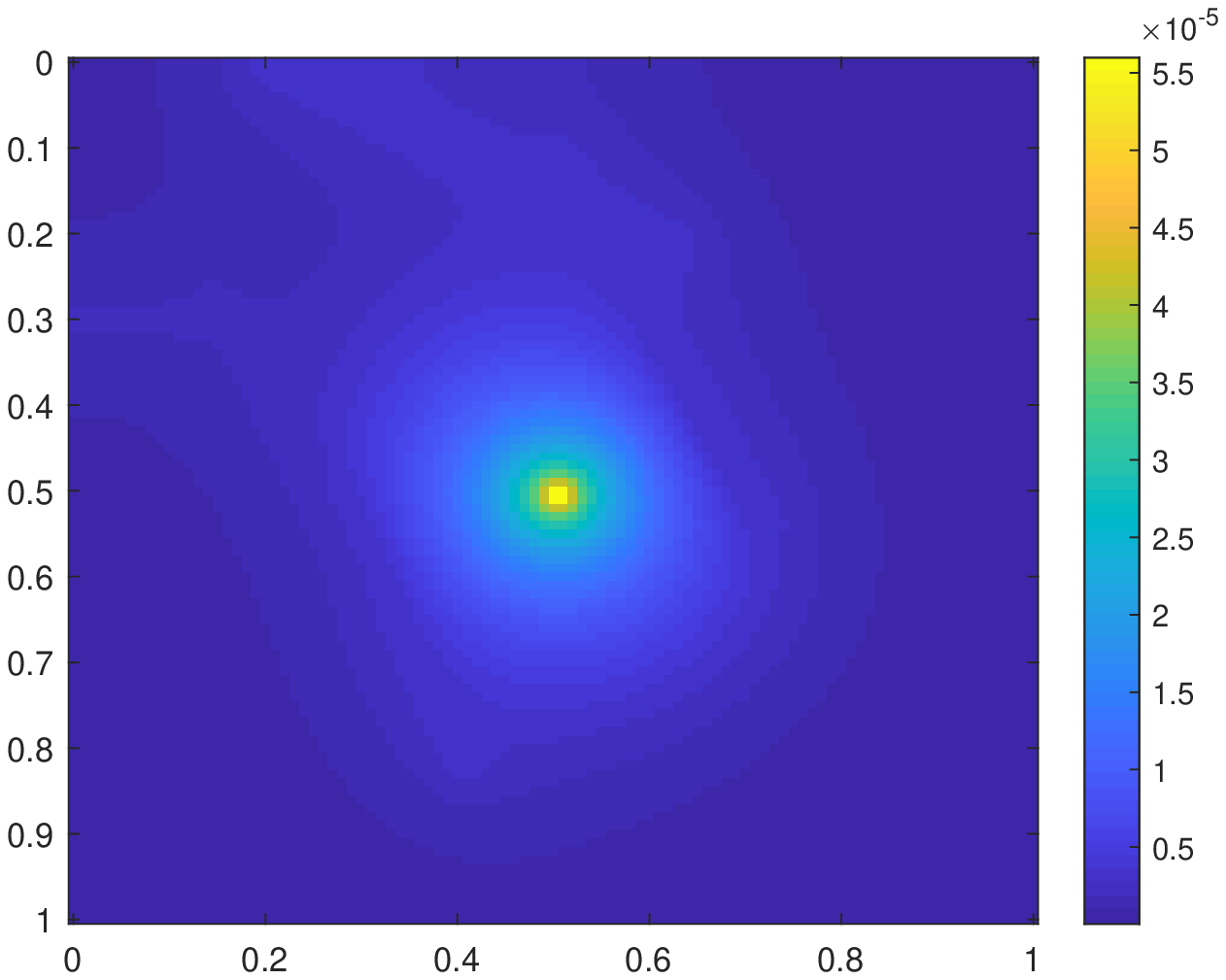}
\includegraphics[scale=0.35]{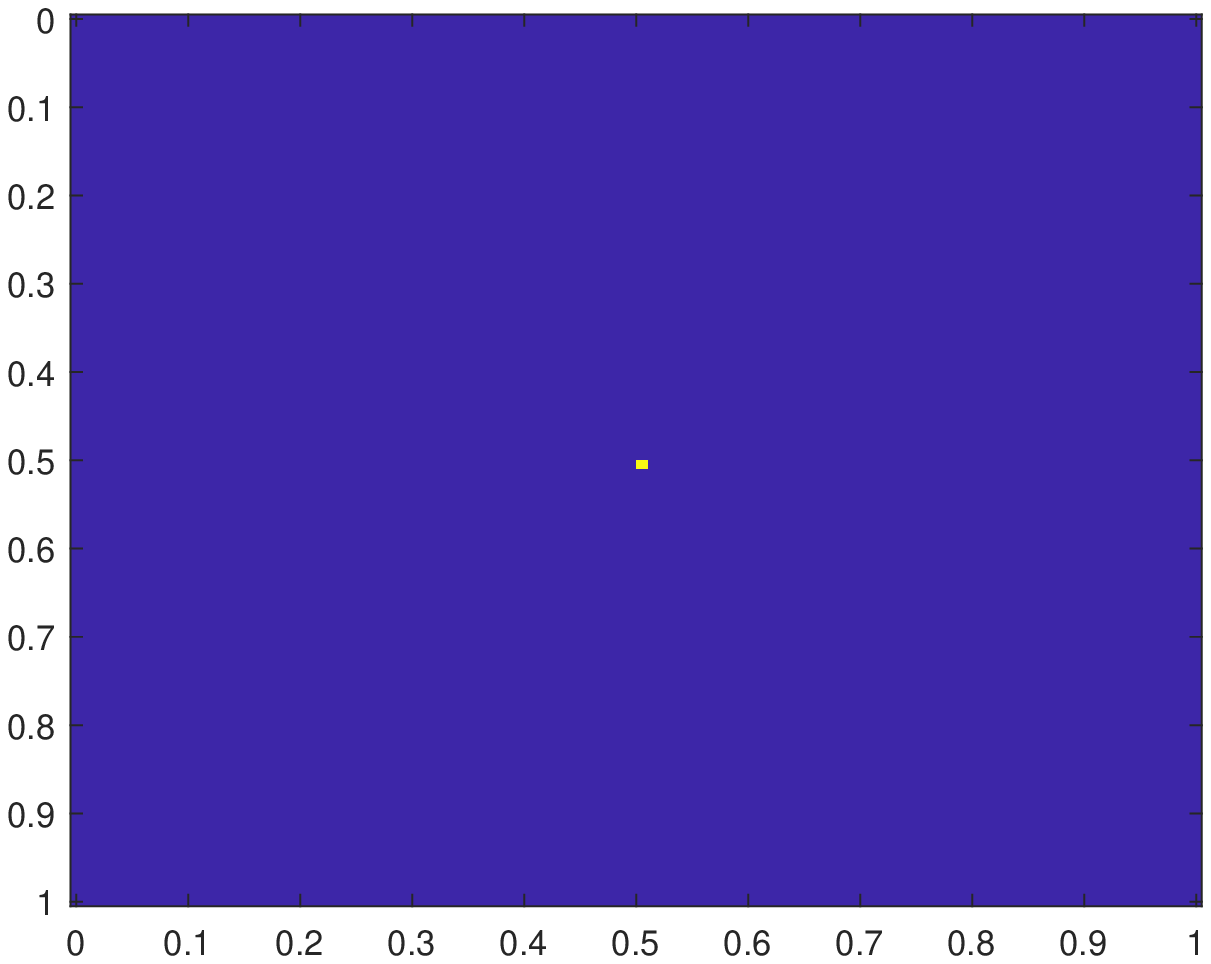}
\caption{Left: $\kappa$. Middle: reference solution at the final time. Right: $f$.}
\label{fig:perms1}
\end{figure}

\begin{figure}[H]
\centering

\includegraphics[scale=0.4]{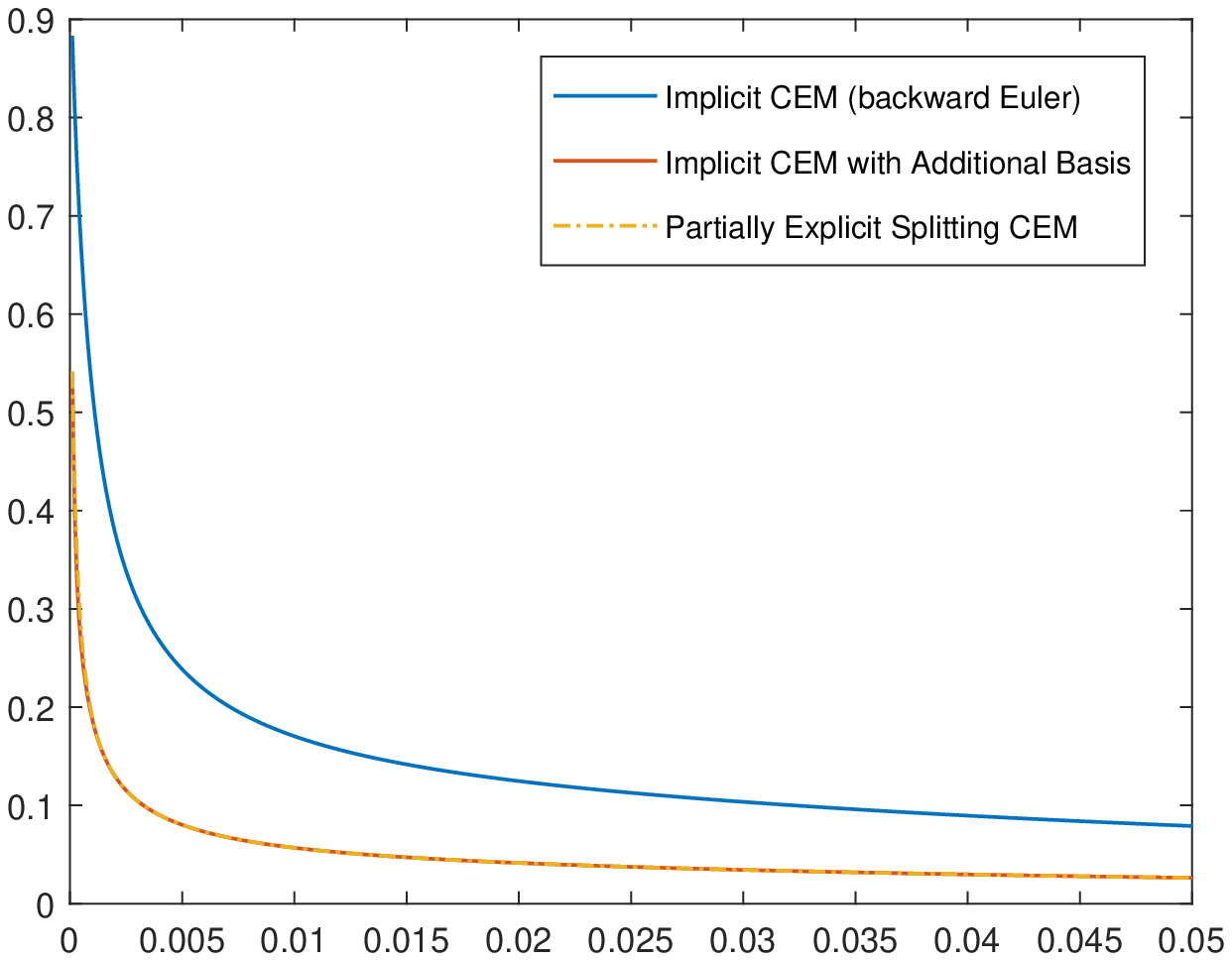} \includegraphics[scale=0.4]{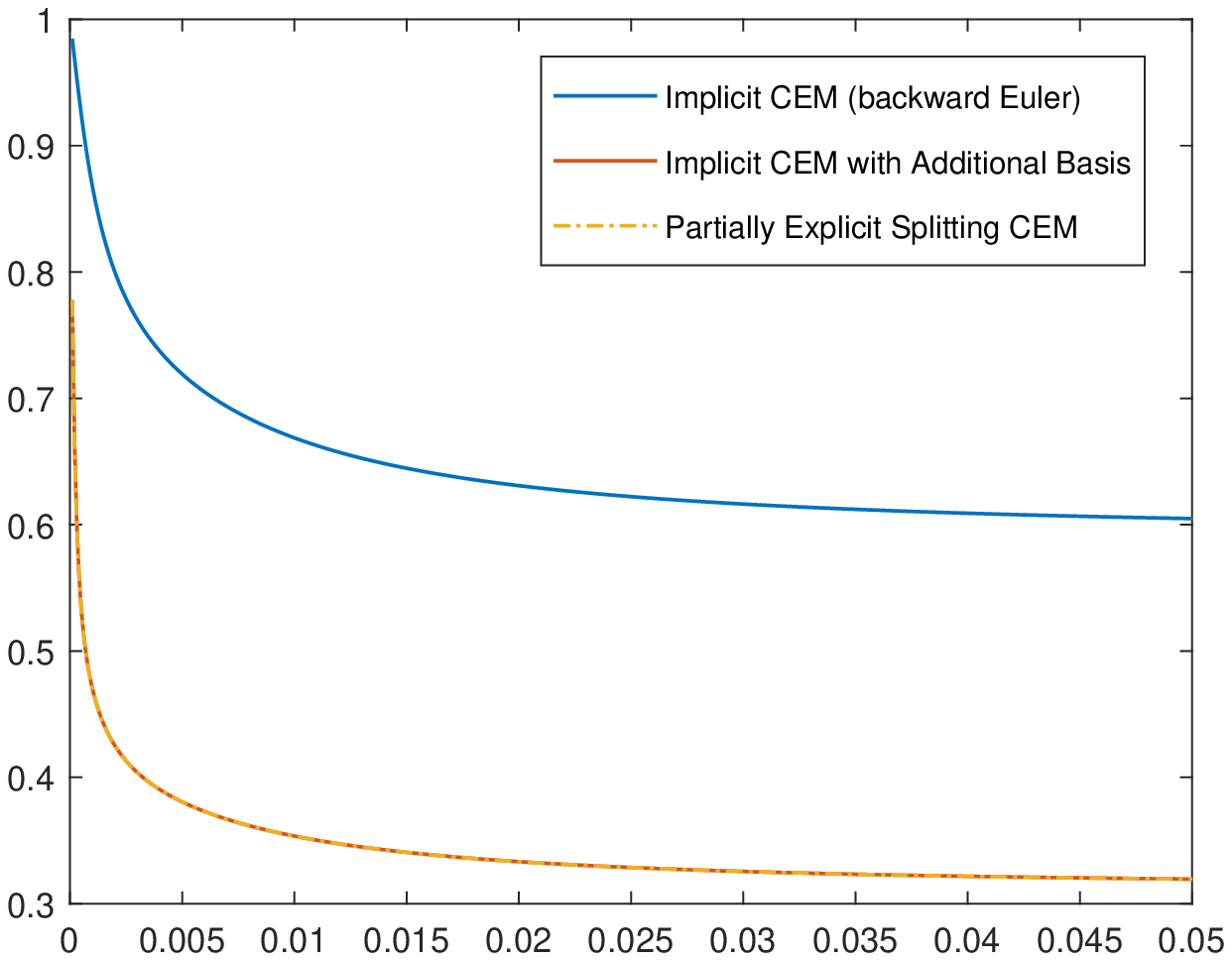}

\caption{Example 2. First type of $V_{2,H}$ (CEM Dof: $300$, $V_{2,H}$ Dof:
$243$). Left: $L_{2}$ error. Right: Energy error.  Along $x$-axis is time, along $y$-axis is the relative error.}
\label{fig:results1.1}
\end{figure}

\begin{figure}[H]
\centering

\includegraphics[scale=0.4]{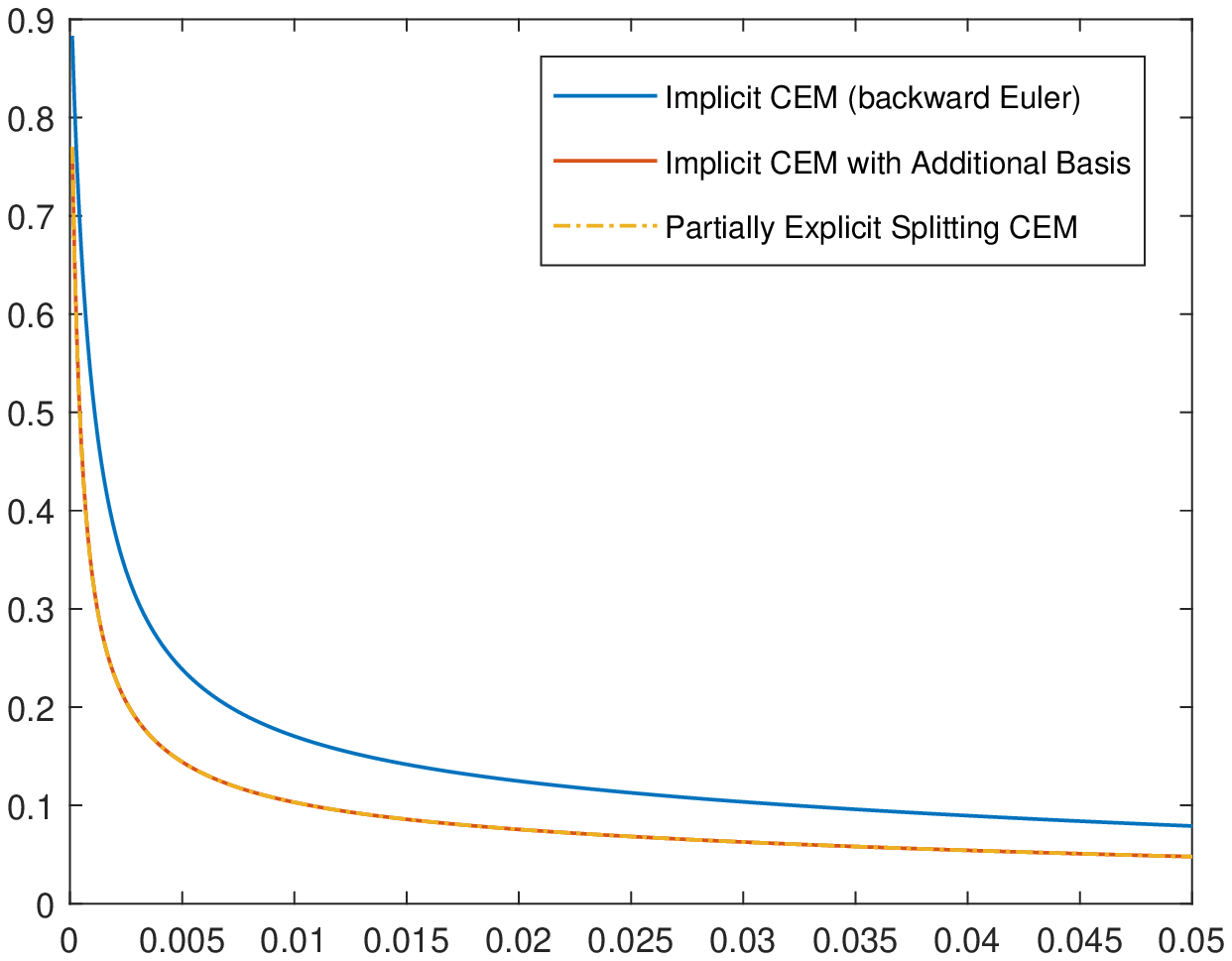} \includegraphics[scale=0.4]{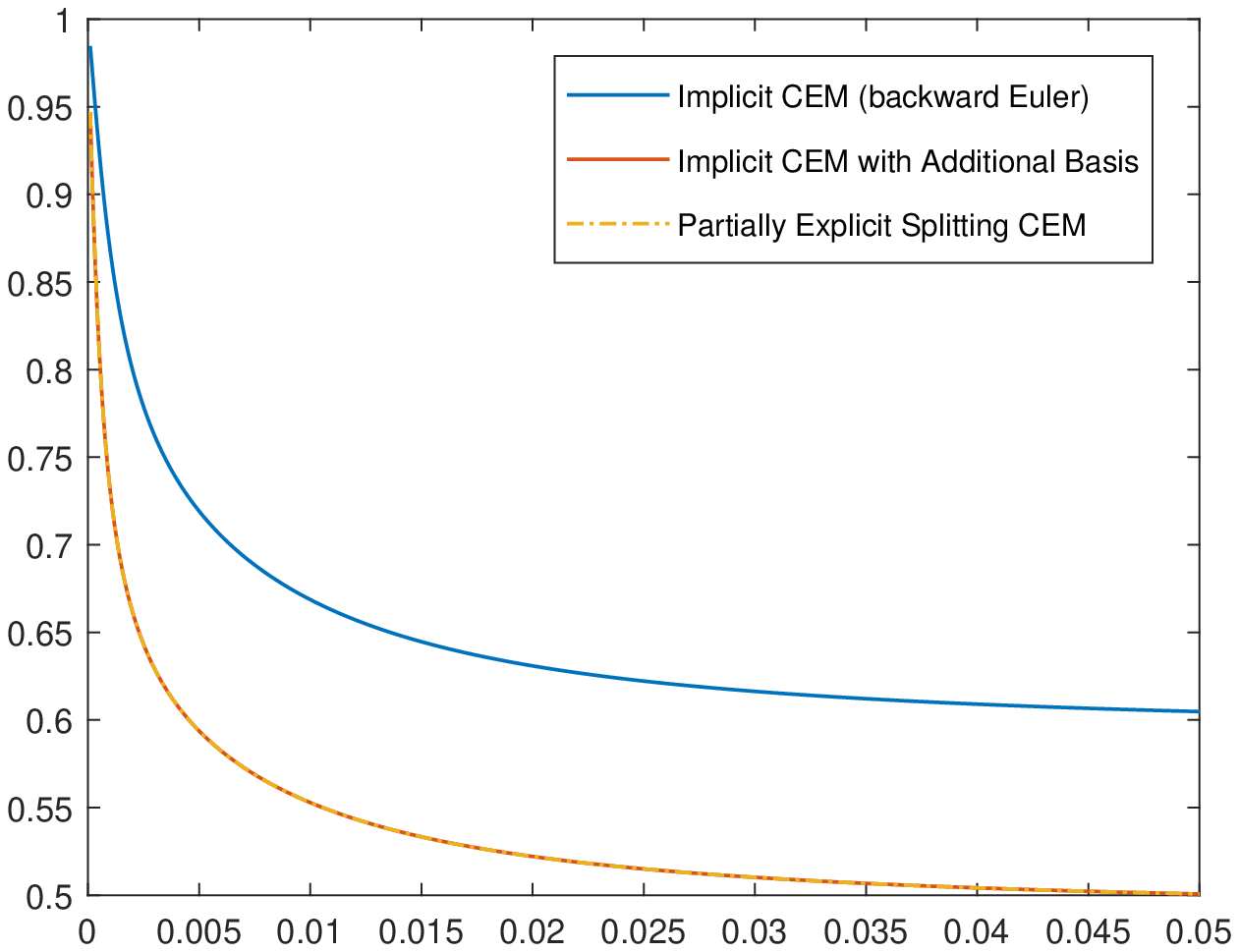}

\caption{Example 2. Second type of $V_{2,H}$ (CEM Dof: $300$, $V_{2,H}$
Dof: $300$). Left: $L_{2}$ error. Right: Energy error.  Along $x$-axis is time, along $y$-axis is the relative error.}
\label{fig:results1.2}
\end{figure}

\subsection*{Numerical Example 3}

For our final numerical test, we take more complicated permeability field
as shown in Figure \ref{fig:perms4} 
(more high conductivity streaks). In this figure,
we also depict  the reference solution at final time
$u_{ref}$, and the source term $f$ are shown in the following figure.
Because of 
a singular source term, as before, CEM-GMsFEM
 with additional basis functions
can give a noticeable improvement as original multiscale CEM basis
functions do not take into account singular source term. 
First, we numerically compute the constant from (\ref{eq:assumption}) that is
assumed to be contrast independent.  
The result is shown in Table \ref{tab:constant2}. As
we see from this table that as we increase the contrast, the constant
remains constant, which asserts that our assumption is true. 
Next, present numerical results.

Next, we present numerical results for two types of $V_{2,H}$, as before.
 We
briefly describe one of them as the results are similar.
In Figure \ref{fig:results4.1} (and Figure \ref{fig:results4.2}
for second type $V_{2,H}$), we present the errors ($L_2$ and energy 
errors) for CEM without additional basis functions (blue curve)
and CEM-GMsFEM with additional basis functions. 
We show CEM-GMsFEM with additional basis functions that use fully 
implicit setting and CEM-GMsFEM with those additional basis functions
that use partially explicit setting coincide.

\begin{figure}[H]
\centering

\includegraphics[scale=0.35]{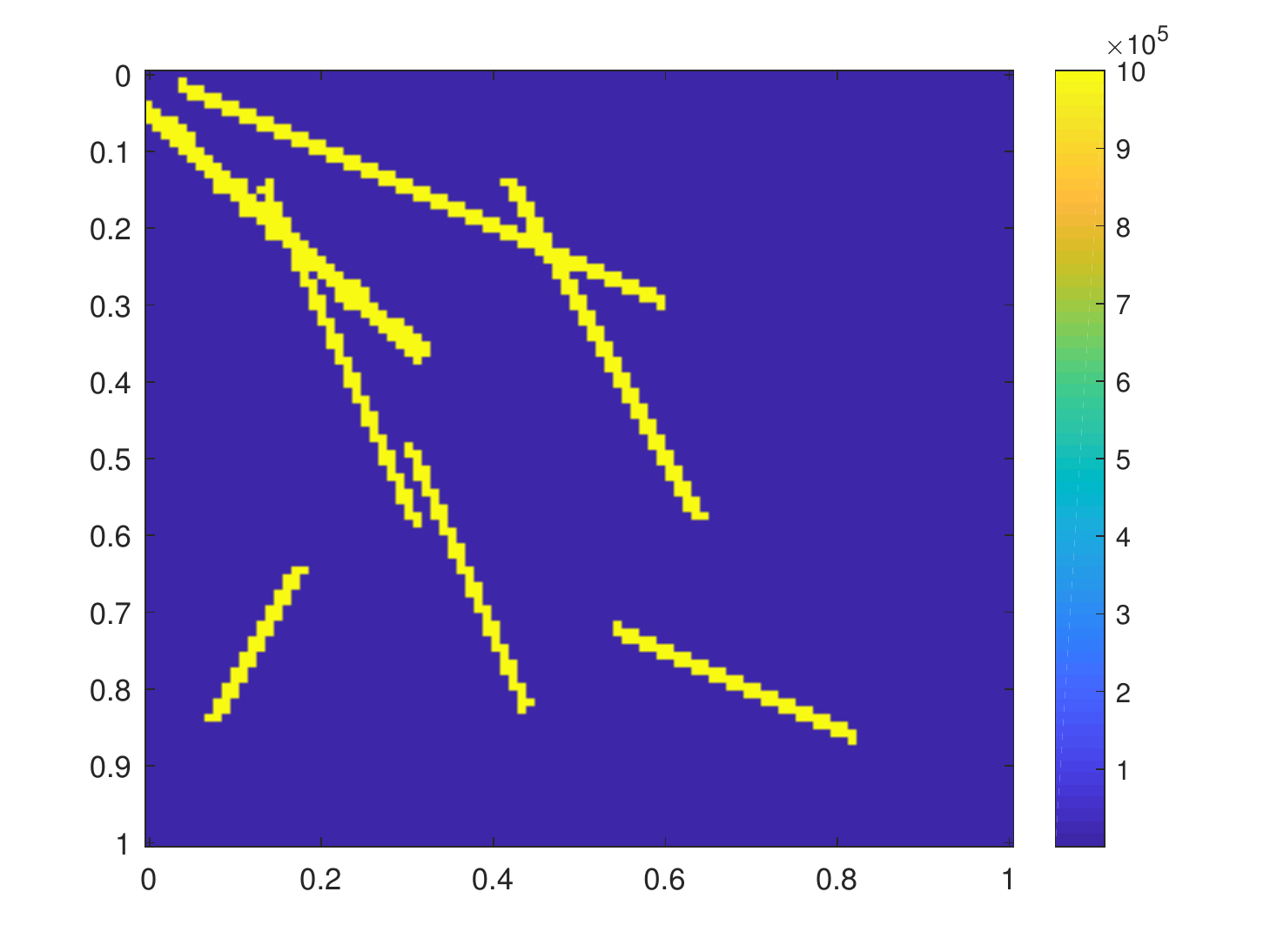} 
\includegraphics[scale=0.35]{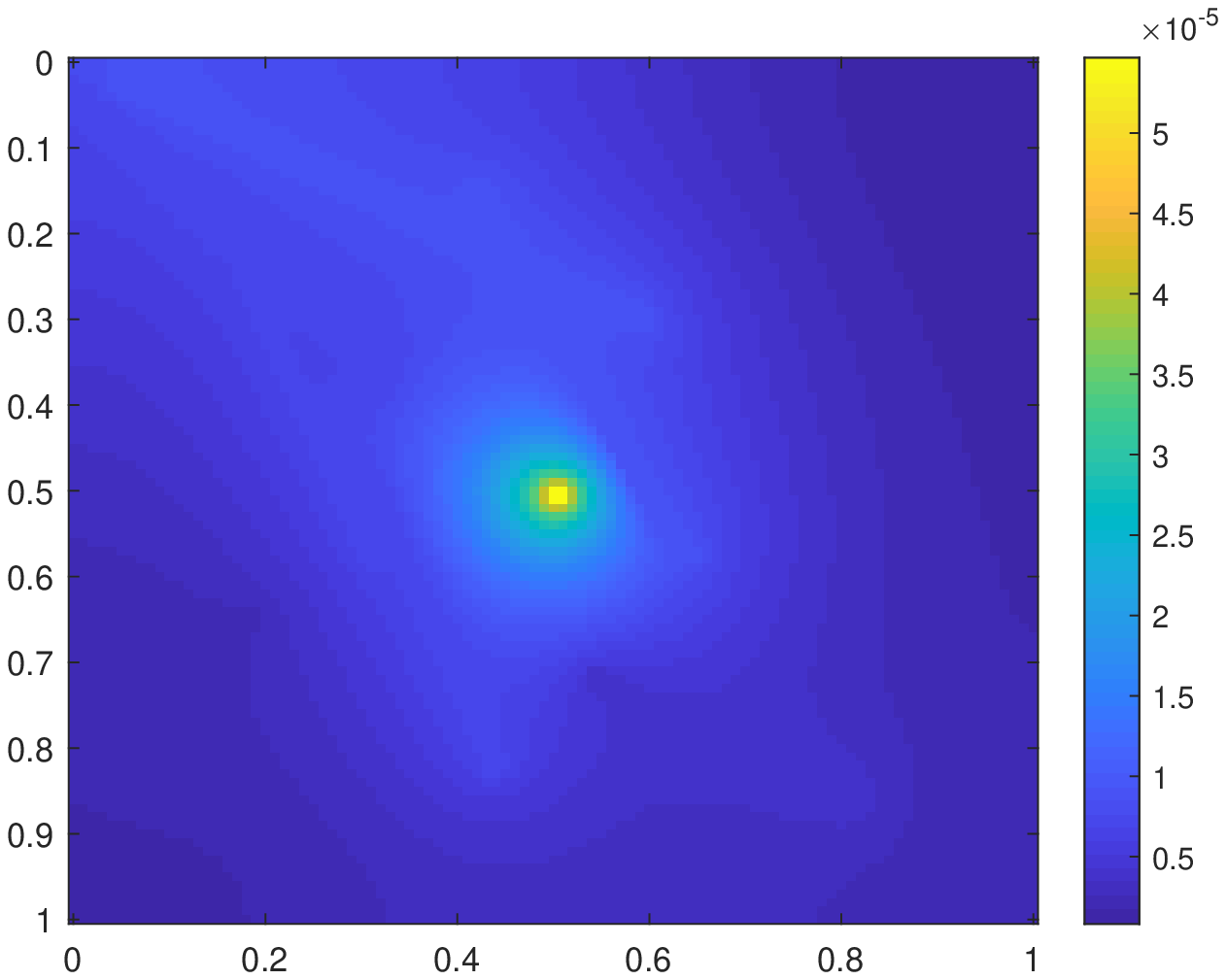}
\includegraphics[scale=0.35]{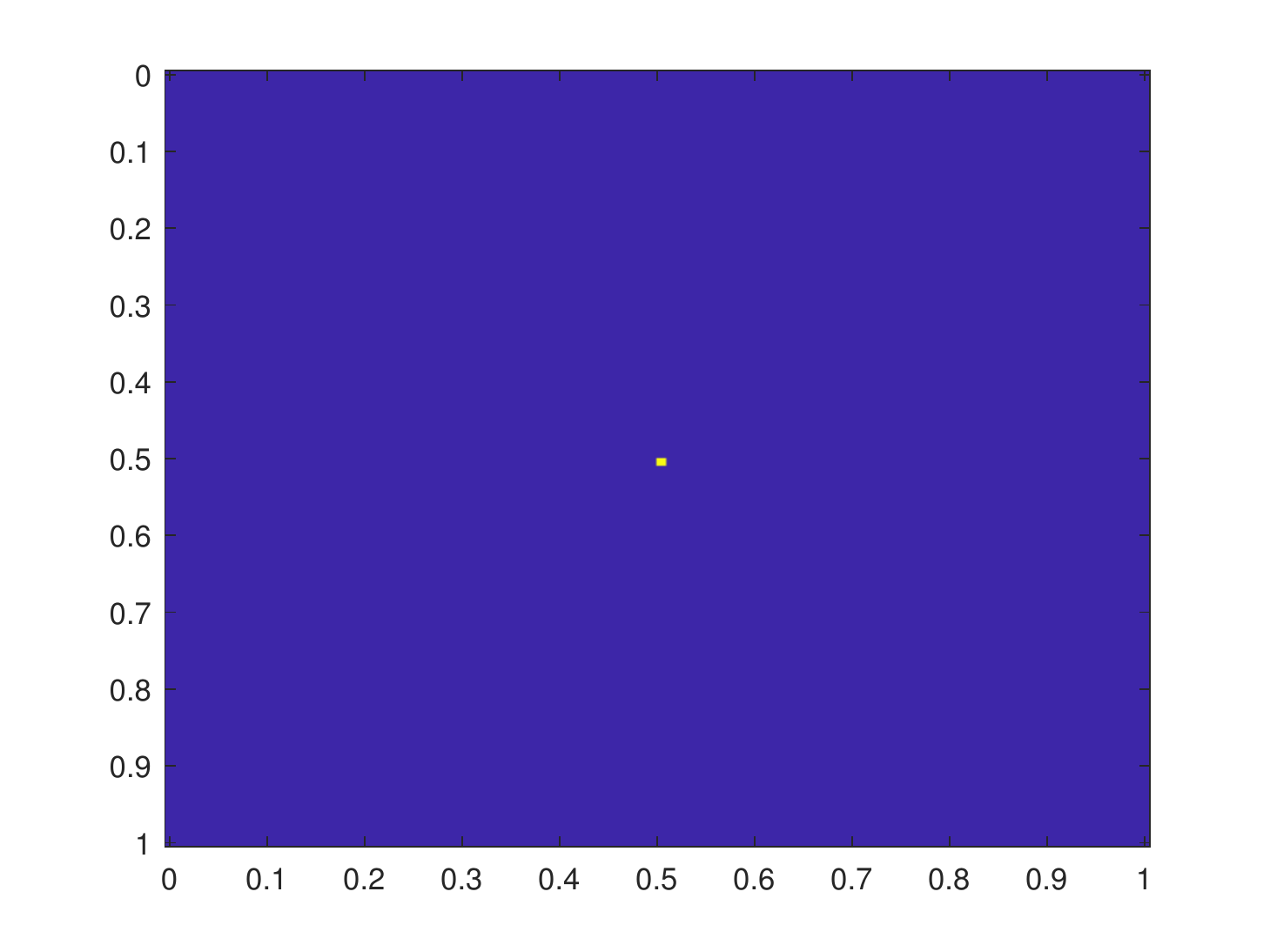}

\caption{Left: $\kappa$, Middle: reference solution at the final time. Right: $f$.}
\label{fig:perms4}
\end{figure}

\begin{table}[H]
\begin{tabular}{|c|c|c|c|c|c|}
\hline 
$\cfrac{\max\{\kappa\}}{\min\{\kappa\}}$ & $10^{5}$ & $10^{6}$ & $10^{7}$ & $10^{8}$ & $10^{9}$\tabularnewline
\hline 
$\sup_{v\in V_{1,H}} \mathcal{G}(v)$ & $1.13\times10^{6}$ & $1.13\times10^{7}$ & $1.13\times10^{8}$ & $1.13\times10^{9}$ & $1.13\times10^{10}$\tabularnewline
\hline 
First type of $V_{2,H}$. $\sup_{v\in V_{2,H}} \mathcal{G}(v)$ & $1.78\times10^{2}$ & $1.78\times10^{2}$ & $1.76\times10^{2}$ & $1.76\times10^{2}$ & $1.76\times10^{2}$\tabularnewline
\hline 
Second type of $V_{2,H}$. $\sup_{v\in V_{2,H}} \mathcal{G}(v)$ & $1.35\times10^{2}$ & $1.35\times10^{2}$ & $1.35\times10^{2}$ & $1.35\times10^{2}$ & $1.35\times10^{2}$\tabularnewline
\hline 
\end{tabular}

\caption{Example 2. $\sup\cfrac{\|v\|_{a}^{2}}{\|v\|_{L^{2}}^{2}}$ for different
$\cfrac{\max\{\kappa\}}{\min\{\kappa\}}$. Here, we denote $\mathcal{G}(v)=\cfrac{\|v\|_{a}^{2}}{H^{-2}\|v\|_{L^{2}}^{2}}$.}
\label{tab:constant2}
\end{table}

\begin{figure}[H]
\centering

\includegraphics[scale=0.4]{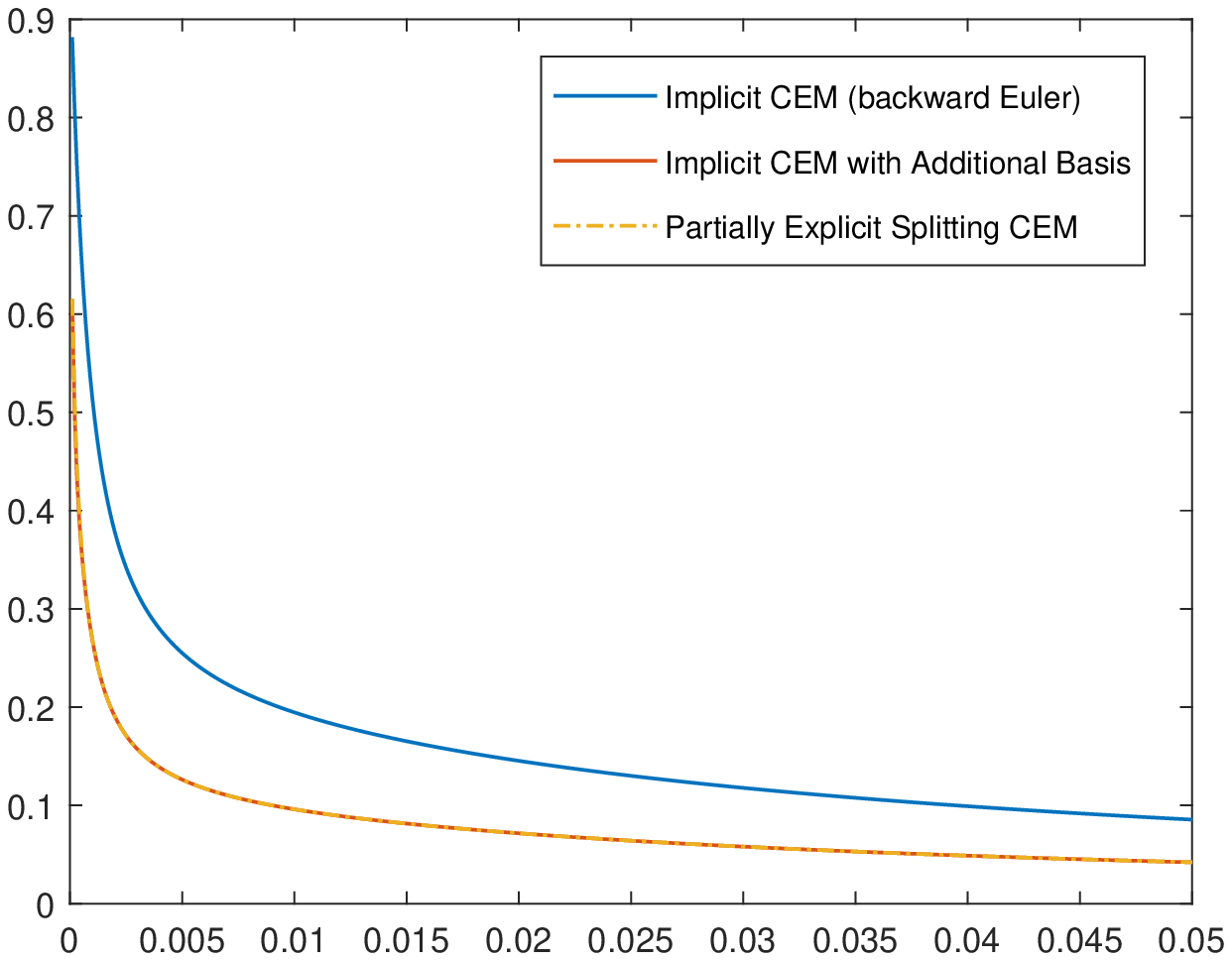} \includegraphics[scale=0.4]{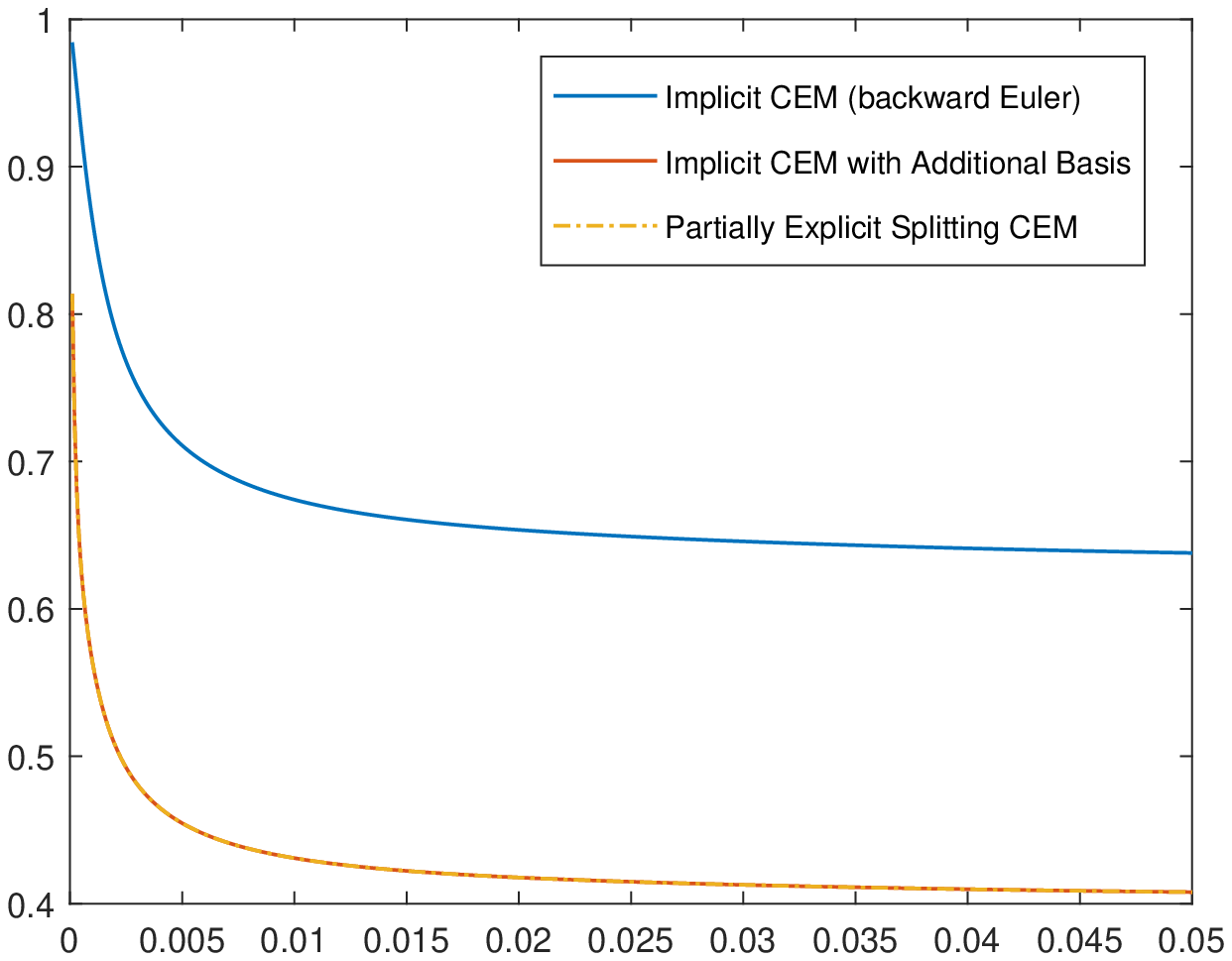}

\caption{Example 3. First type of $V_{2,H}$ (CEM Dof: $300$, $V_{2,H}$ Dof:
$243$). Left: $L_{2}$ error. Right: Energy error.  Along $x$-axis is time, along $y$-axis is the relative error.}
\label{fig:results4.1}
\end{figure}

\begin{figure}[H]
\centering

\includegraphics[scale=0.4]{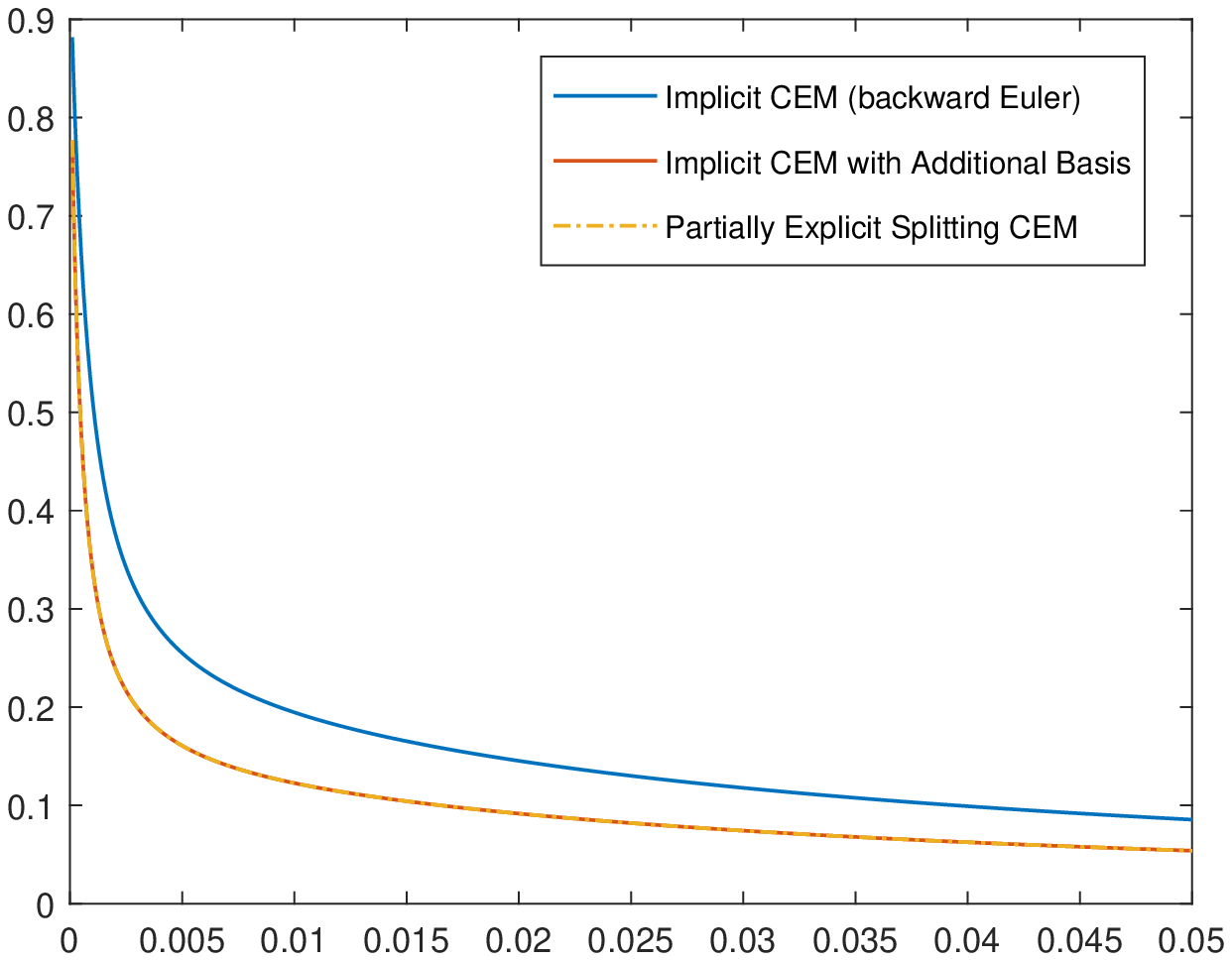} \includegraphics[scale=0.4]{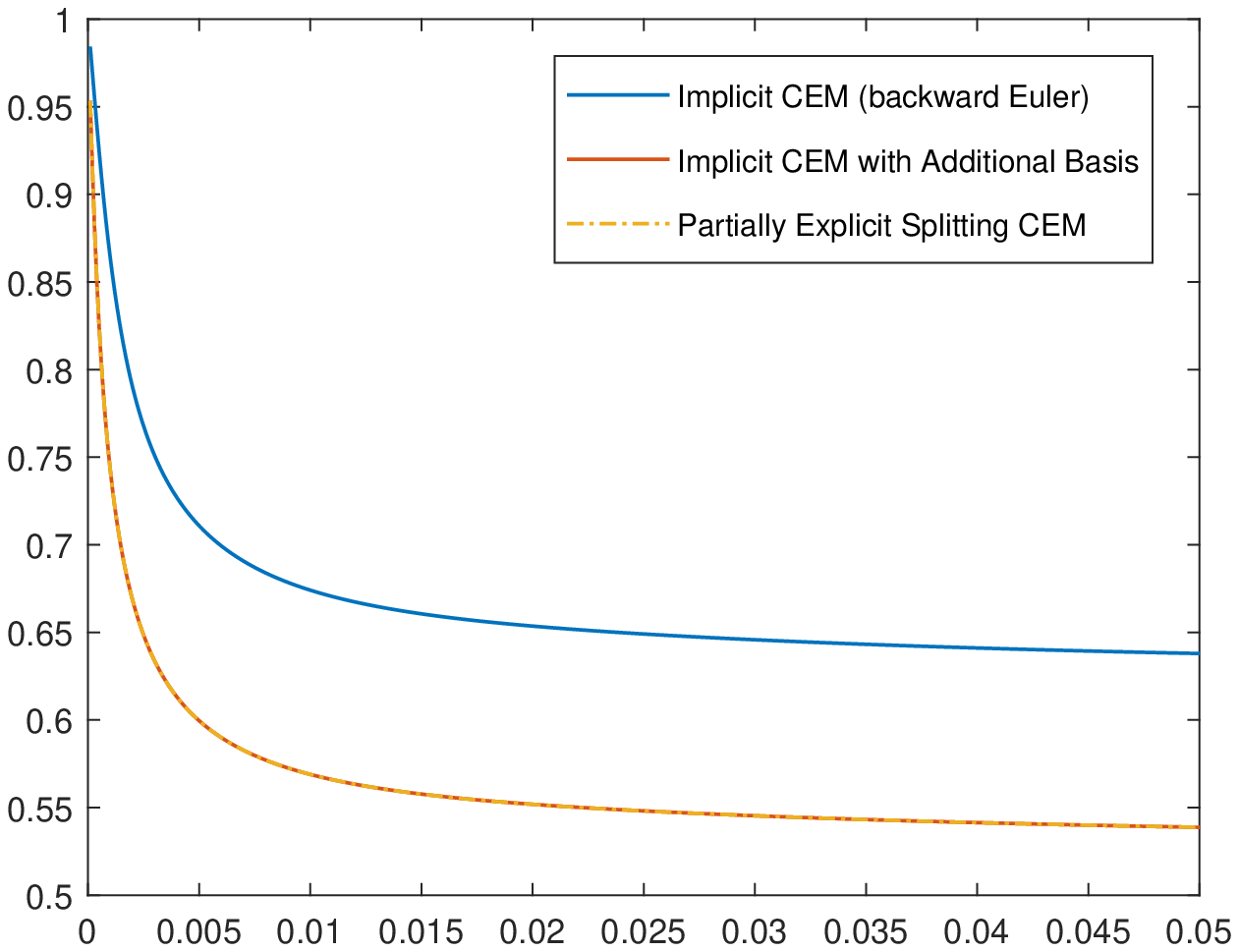}

\caption{Example 2. Second type of $V_{2,H}$ (CEM Dof: $300$, $V_{2,H}$
Dof: $300$). Left: $L_{2}$ error. Right: Energy error.  Along $x$-axis is time, along $y$-axis is the relative error.}
\label{fig:results4.2}
\end{figure}

\section{Conclusions}

In this paper, we study the development of temporal discretizations
that can use time stepping independent of the contrast. We consider
a parabolic equation, where the coefficient is multiscale and have high
contrast features. We propose a partially explicit method, where the proposed
method is stable with the time step that doesnt depend on the contrast.
The development of the proposed method requires special multiscale basis
construction and temporal splitting.
 Our coarse space consists of CEM-GMsFEM basis functions
and special multiscale basis functions for the remaining degrees of freedom
are constructed. The coarse-grid component of the solution 
(that has a few degrees of freedom) is solved
implicitly with explicit contributions from the rest. The remaining part
is updated in an explicit fashion within proposed splitting algorithms.
 We show that the resulting approach 
is stable with the time step is independent of the contrast. 
Appropriate multiscale decomposition of the space is needed
for the success of the approach as shown in the paper. 
We formulate sufficient conditions for the decomposition and construct
appropriate spatial decomposition. We present numerical results.
Our numerical results show that the proposed partial explicit methods
give almost the same accuracy as fully implicit method.

\section*{Acknowledgement}

The research of Eric Chung is partially supported by the Hong Kong RGC General Research Fund (Project numbers 14304719 and 14302018) and the CUHK Faculty of Science Direct Grant 2019-20.

\appendix\section{Motivation for $V_{H,2}$ based on approximation errors}

In this appendix, we discuss some motivations of the choices of $V_{H,2}$ based on error reduction viewpoint. 

\subsection{First choice}

We consider the first choice presented in Section \ref{sec:choice1}. 
For simplicity, we let $V_{H,1} = V_{glo}$, where $V_{glo}$ is the CEM space defined in (\ref{eq:glospace}). 
We consider the elliptic problem: find $u\in V$ such that 
\begin{equation*}
a(u,v) = (f,v), \quad\forall v\in V.
\end{equation*}
The corresponding multiscale problem is: find $u_H := u_{H,1}+u_{H,2} \in V_{H,1}+V_{H,2}$ such that
\begin{equation*}
a(u_H,v) = (f,v), \quad\forall v\in V_{H,1}+V_{H,2}.
\end{equation*}
Subtracting the above two equations, we obtain
\begin{equation*}
a(u-u_H,v) = 0, \quad \forall v\in V_{H,1}+V_{H,2}.
\end{equation*}
Recall that $V = V_{H,1} + \tilde{V}$,
$V_{H,1}$ and $\tilde{V}$ are $a$-orthogonal, and that $V_{H,2}\subset \tilde{V}$. So, we have
\begin{equation*}
a(u-u_{H,1},v) = 0,\quad \forall v\in V_{H,1},
\end{equation*}
which implies that $u-u_{H,1} \in \tilde{V}$. Taking the test function $v\in V_{H,2}$, we obtain
\begin{equation*}
a(u_{H,2},v) = a(u-u_{H,1}, v), \quad \forall v\in V_{H,2}.
\end{equation*}
From this equation, we see that $V_{H,2}$ provides a correction of the solution $u_{H,1}$ based on the residual
$a(u-u_{H,1},v)$.

To derive an error bound, we note that
\begin{equation*}
\|u-u_{H,1}\|_a^2 = a(u-u_{H,1},u-u_{H,1}) = a(u-u_{H,1},u) = (f,u-u_{H,1}).
\end{equation*}
Using the definition of the $s$-norm defined in Section \ref{sec:cem}, we have
\begin{equation*}
(f,u-u_{H,1}) \leq CH \| \kappa^{-\frac{1}{2}} f\| \|u-u_{H,1}\|_s.
\end{equation*}
Using the fact that $u-u_{H,1} \in \tilde{V}$ and the spectral problem (\ref{eq:spectralCEM}), we obtain 
\begin{equation*}
 \|u-u_{H,1}\|_s^2 \leq ( \min_i  \lambda_{L_i+1}^{(i)} )^{-1} \| u-u_{H,1}\|_a^2.
\end{equation*}
Combining the results, we obtain the following energy norm error bound
\begin{equation*}
\| u - u_{H,1}\|_a \leq CH( \min_i  \lambda_{L_i+1}^{(i)} )^{-\frac{1}{2}} \| \kappa^{-\frac{1}{2}} f\|.
\end{equation*}
To get a $L^2$ error bound, we consider the dual problem: given $g\in L^2(\Omega)$, 
 find $z\in V$ such that 
\begin{equation*}
a(v,z) = (g,v), \quad\forall v\in V.
\end{equation*}
The corresponding multiscale problem is given by: find $z_{H} \in V_{H,1}$ such that 
\begin{equation*}
a(v,z_H) = (g,v), \quad\forall v\in V_{H,1}.
\end{equation*}
Let $g = \kappa (u-u_{H,1})$. We have
\begin{equation*}
\begin{split}
\| \kappa^{\frac{1}{2}} (u-u_{H,1}) \|^2
&= (g, u-u_{H,1}) = a(u-u_{H,1},z)
= a(u-u_{H,1}, z-z_H) \\
&\leq \|u-u_{H,1}\|_a \|z-z_H\|_a
\leq C^2H^2 ( \min_i  \lambda_{L_i+1}^{(i)} )^{-1} \| \kappa^{-\frac{1}{2}} f\| \| \kappa^{\frac{1}{2}} (u-u_{H,1})\|.
\end{split}
\end{equation*}
So, we obtain
\begin{equation*}
\| \kappa^{\frac{1}{2}} (u-u_{H,1}) \| \leq C^2H^2 ( \min_i  \lambda_{L_i+1}^{(i)} )^{-1} \| \kappa^{-\frac{1}{2}} f\|.
\end{equation*}

Now, we derive the full error $\|u-(u_{H,1}+u_{H,2})\|_a$. Note that
$u_{H,2}$ is the $a$-orthogonal projection of $u-u_{H,1}$ in the space $V_{H,2}$. So,
\begin{equation*}
\|u-(u_{H,1}+u_{H,2})\|_a \leq \| u - u_{H,1} - v\|_a,  \quad\forall v\in V_{H,2}.
\end{equation*}
Assume the domain $\Omega$ is rectangular and the mesh $\mathcal{T}_H$ is a regular grid. 
Let $\{\chi_i\}$ be a set of smooth partition of unity functions corresponding to the overlapping partition $\cup \{ \omega_i\}$ of $\Omega$ with $|\nabla\chi_i | \leq CH^{-1}$, 
and that each $\chi_i$ has zero trace on $\partial \omega_i$. We write
\begin{equation*}
u-u_{H,1} = \sum_i \chi_i (u-u_{H,1}) = \sum_i r_i
\end{equation*}
where $r_i = \chi_i (u-u_{H,1})$.
We assume that the eigenfunctions of (\ref{eq:eigenvalueproblem_case2}) forms a complete basis, so that each $r_i$ can be represented by
\begin{equation*}
r_i = \sum_{j} a(r_i, \xi_j^{(i)}) \xi_j^{(i)}
\end{equation*}
where we also assume the normalized condition $a(\xi_j^{(i)}, \xi_j^{(i)})=1$. We define
\begin{equation*}
v :=  \sum_i v_i := \sum_i \sum_{j \leq J_i} a(r_i, \xi_j^{(i)}) \xi_j^{(i)} \in V_{H,2}.
\end{equation*}
So, we have
\begin{equation*}
\|u-(u_{H,1}+u_{H,2})\|^2_a \leq 4 \sum_i  \| r_i - v_i\|_a^2.
\end{equation*}
Define $0\leq\theta \leq 1$ by
\begin{equation*}
\theta = \max_i \cfrac{\| r_i - v_i\|_a}{\| r_i\|_a}
\end{equation*}
which represents relative reduction of error. 
Notice that
\begin{equation*}
\sum_i \|r_i\|_a^2 \leq 8 \| u-u_{H,1}\|_a^2 + 8 \| u-u_{H,1}\|_s^2 \leq 
16 C^4 H^2 ( \min_i  \lambda_{L_i+1}^{(i)} )^{-2} \| \kappa^{-\frac{1}{2}} f\|^2.
\end{equation*}
Combining all results, we obtain the following error bound
\begin{equation*}
\|u-(u_{H,1}+u_{H,2})\|_a \leq C_0 \theta H( \min_i  \lambda_{L_i+1}^{(i)} )^{-1} \| \kappa^{-\frac{1}{2}} f\|.
\end{equation*}

\subsection{Second choice}

We consider the second choice in this section. 
We consider an estimate of the elliptic projection
of $V_{glo}+V_{glo,2}$ with second choice of $V_{glo,2}$. Specifically, we assume
$u_{1}\in V_{glo}$, $u_{2}\in V_{glo,2}$ and $u\in V$ satisfy
\[
a(u,v)=(f,v),\;\;\forall v\in V,
\]
and
\[
a(u_{1}+u_{2},v_{1}+v_{2})=(f,v_{1}+v_{2}), \;\;\forall v_{1}\in V_{glo},v_{2}\in V_{glo,2}.
\]
We define $\overline{V}=\{v\in V|\;\tilde{\Pi}(v)=0\}$ and we can easy
check that $\overline{V}$ is $a$-orthogonal to $V_{glo,1}+V_{glo,2}$,
namely $(V_{glo}+V_{glo,2})\subset\overline{V}^{\perp_{a}}$.
where $\overline{V}^{\perp_{a}}$ is the orthogonal complement of $\overline{V}$ with respect to the $a(\cdot,\cdot)$ inner product.
By counting
the dimension of $\overline{V}^{\perp_{a}}$ and $V_{glo}+V_{glo,2}$,
we have 
\[
(V_{glo,1}+V_{glo,2})=\overline{V}^{\perp_{a}}\text{ and }(V_{glo}+V_{glo,2})^{\perp_{a}}=\overline{V}.
\]
Since
\[
a(u-u_{1}-u_{2},v)=0,\;\;\forall v\in V_{glo}+V_{glo,2},
\]
we have 
\[
u-u_{1}-u_{2}\in(V_{glo}+V_{glo,2})^{\perp_{a}}=\overline{V}.
\]
Thus, we have 
\begin{align*}
a(u-u_{1}-u_{2},u-u_{1}-u_{2}) & =a(u,u-u_{1}-u_{2})=(f,u-u_{1}-u_{2})\\
 & \leq\|f\| \|(u-u_{1}-u_{2})\|.
\end{align*}
Note that $\overline{V} \subset \tilde{V}$. 
Since $u-u_{1}-u_{2}\in\overline{V}$, we can write $u-u_{1}-u_{2}$ in terms of the eigenfunctions of (\ref{eq:spectralCEM2}),
 \[
 u-u_{1}-u_{2} = \sum_{i=1}\sum^{\infty}_{j=1} a^{(i)}_j\xi^{(i)}_j.
 \]
 Since $u-u_{1}-u_{2}\in\overline{V} \subset \tilde{V}$, we have
 \[
 a^{(i)}_j =0 \;\forall j\leq J_i,
 \]
 which implies
 \[
 u-u_{1}-u_{2} = \sum_{i=1}\sum^{\infty}_{j=J_i+1} a^{(i)}_j\xi^{(i)}_j.
 \]
 So, we have
 \begin{align*}
H^{-2}\|(u-u_{1}-u_{2})\|^2_{L^{2}} & =\sum_{i=1}\sum^{\infty}_{j=J_i+1} (a^{(i)}_j)^2\leq \sum_{i=1}(\gamma^{(i)}_{J_i+1})^{-1}\sum^{\infty}_{j=J_i+1} \gamma^{(i)}_j(a^{(i)}_j)^2\\
 & \leq (\min_i \gamma^{(i)}_{J_i+1})^{-1}\sum_{i=1}\sum^{\infty}_{j=J_i+1} \gamma^{(i)}_j(a^{(i)}_j)^2\\
 &\leq (\min\{\gamma_{J_{i}+1}^{(i)}\})^{-1}\|u-u_{1}-u_{2}\|^2_{a}
\end{align*}
and
\begin{align*}
\|(u-u_{1}-u_{2})\|_{L^{2}} &  \leq\cfrac{H}{(\min\{\gamma_{J_{i}+1}^{(i)}\})^{\frac{1}{2}}}\|u-u_{1}-u_{2}\|_{a}.
\end{align*}
Therefore, we have 
\[
\|u-u_{1}-u_{2}\|_{a}\leq\cfrac{H\|f\|_{L^{2}}}{(\min\{\gamma_{J_{i}+1}^{(i)}\})^{\frac{1}{2}}}.
\]
We remark that for the standard CEM method, we have the error estimate
\[
\|u-u_{1}-u_{2}\|_{a}\leq\cfrac{H\|\kappa^{-\frac{1}{2}}f\|_{L^{2}}}{(\min\{\lambda_{L_{i}+1}^{(i)}\})^{\frac{1}{2}}}.
\]

By the definition of the $s$-norm, we have
\[
\|v\|^2_s \geq C\kappa_{min}H^{-2}\|v\|^2
\]
where we use the fact that $\sum_i|\nabla \chi_i|^2\leq CH^{-2}$. So,
\[
\gamma_{J_{i}+1}^{(i)}\geq\gamma_{1}^{(i)}
=\min_{v\in\tilde{V}}\cfrac{\|v\|_{a}^{2}}{H^{-2}\|v\|_{L^{2}}^{2}}
 \geq C\kappa_{min}\min_{v\in\tilde{V}}\cfrac{\|v\|_{a}^{2}}{\|v\|_{s}^{2}} 
=C\kappa_{min}\lambda_{L_{i}+1}^{(i)}.
\]
Thus, the enriched space $V_{H,2}$ can improve the elliptic projection error from $O(\cfrac{H}{(\min\{\lambda_{L_{i}+1}^{(i)}\})^{\frac{1}{2}}})$ to $O(\cfrac{H}{(\min\{\gamma_{J_{i}+1}^{(i)}\})^{\frac{1}{2}}})$.

\bibliographystyle{abbrv}
\bibliography{references,references4,references1,references2,references3,decSol}

\end{document}